\documentclass[11pt,reqno]{amsart}
\usepackage{enumerate}
\usepackage{enumitem}
\usepackage{mathrsfs}
\usepackage{url}
\usepackage{cite}
\usepackage{comment}
\usepackage{fancyhdr}
\usepackage{amsfonts, amsmath, amssymb, amscd, amsthm, bm, cancel}
\usepackage{amsrefs}
\usepackage[mathlines]{lineno}
\usepackage[
linktocpage=true,colorlinks,citecolor=magenta,linkcolor=blue,urlcolor=magenta]{hyperref}
\usepackage[margin=1.3in]{geometry}
\newtheorem{thm}{Theorem}[section]

\newtheorem{rmk}{Remark}[section]

\newtheorem{lemma}{Lemma}[section]
\newtheorem{corollary}{Corollary}[section]

\numberwithin{equation}{section}
\allowdisplaybreaks

\begin{document}
\title[Support measures and Alexandrov-Fenchel type inequalities]{Modified support measures and Alexandrov-Fenchel type inequalities for nonsmooth h-convex bodies in hyperbolic space}

\author[X. Zhang]{Xinxiu Zhang}
\address{School of Mathematical Sciences, University of Science and Technology of China, Hefei 230026, P.R. China}
\email{\href{mailto:zhangxinxiu@mail.ustc.edu.cn}{zhangxinxiu@mail.ustc.edu.cn}}

\date{\today}
\begin{abstract}    

We establish the Alexandrov-Fenchel inequality and its modified version for arbitrary, possibly nonsmooth, horospherically convex bodies in hyperbolic space $\mathbb{H}^{n+1}, n\geq 2$, and prove that equality in each inequality holds if and only if the h-convex body is a geodesic ball. A key ingredient is a rigidity theorem characterizing geodesic balls via proportionality between modified support measures of adjacent orders. Its proof combines the $C^{1,\alpha}\cap W^{2,k+1}$ approximation of the outer parallel hypersurfaces  by smooth strictly h-convex hypersurfaces with the estimates for the total variation norm of the signed measure associated with the proportionality relation  along the  shifted curvature flows independent of initial curvature bounds. 

 For $0\leq k\leq n-1$, we construct weak flows preserving the $k$-th modified quermassintegral $\tilde{W}_k$ with arbitrary h-convex initial bodies as Hausdorff limits of the smooth flows. The rigidity theorem above characterizes the limit shape of the flow and  gives a positive lower bound on the rate of decrease of $\tilde{W}_{k+1}$ for a short time whenever the initial h-convex body is not a geodesic ball, which characterizes the equality cases in the modified Alexandrov-Fenchel inequality. 

For the Alexandrov-Fenchel inequality, we establish the Hausdorff stability independent of curvature bounds described by the Alexandrov-Fenchel deficits for h-convex bodies, which characterizes the equality cases in the Alexandrov-Fenchel  inequality. 
\end{abstract}
\maketitle
\tableofcontents

\section{Introduction}

The classical Alexandrov-Fenchel inequality for quermassintegrals of convex bodies in Euclidean space was proved by Alexandrov \cite{Alexandrov1937, Alexandrov1937II} and Fenchel \cite{Fenchel1936}.  The concepts of quermassintegrals in Euclidean space can be extended to space forms.
  Recall that a compact domain in hyperbolic space is called an h-convex body if it is the intersection of a class of horoballs with nonempty interior. In the smooth case, a smooth bounded domain is (strictly) h-convex if and only if the principal curvatures are  bounded below by $1$ (greater than $1$).
Wang and Xia \cite{WangXia2014}  obtained the Alexandrov-Fenchel inequality for smooth h-convex bodies in hyperbolic space. More precisely,  they proved 
\begin{align}\label{alexandrov}
 W_k(\Omega) \geq f_k \circ f_{\ell}^{-1}\left(W_{\ell}(\Omega)\right)   
\end{align}
for any smooth h-convex body $\Omega$ and $0\leq \ell<k\leq n$. Here $f_i:[0, \infty) \rightarrow \mathbb{R}_{+}$ is a monotone
function defined by $f_i(r)=W_i\left(B_r\right)$, and equality holds if and only if $\Omega$ is a geodesic ball.

Andrews, Chen and Wei \cite{AndrewsChenWei2021} introduced the concepts of modified quermassintegrals of  smooth h-convex bodies. More precisely, suppose that $\Omega$ is a smooth h-convex body and denote by 
\begin{align}\tilde{W}_k(\Omega):=\sum_{i=0}^k(-1)^{k-i}\binom{k}{i} W_i(\Omega)\label{Quermassine}\end{align} the modified quermassintegrals of $\Omega$. They considered the evolution of a closed smooth strictly h-convex initial hypersurface $M=\partial \Omega$ along the modified quermassintegral shifted curvature flow 
\begin{align}
\left\{\begin{aligned} 
\frac{\partial}{\partial t} X(x, t) & =\left(\phi(t)-\frac{\tilde{E}_{k+1}}{\tilde{E}_k}(W)\right) \nu(x, t), \\ 
X(\cdot,0) & =X_{0}(\cdot),\end{aligned}\right.\label{smoothflow1}
\end{align}
where  $X_{0}:N^n\rightarrow M=\partial\Omega$, $\nu(x, t)$ is the unit outer normal of the hypersurface $M_{t}:=X(N^n,t)=\partial \Omega_{t}$, $W$ is the Weingarten matrix of $M_{t}$, $\tilde{E}_k(x,t):=E_k(\mathcal{W}-I)$ is the shifted  $k$-th mean curvature of $M_{t}$ and  \begin{align}
    \phi(t):=\frac{\int_{M_{t}}\tilde{E}_{k+1}(\mathcal{W}) \mathrm{d} \mu }{\int_{M_{t}} \tilde{E}_k(W)d\mu} \label{phi}
\end{align} 
    is the global term chosen to preserve the $k$-th modified quermassintegral of the enclosed domains. They obtained the Alexandrov-Fenchel type inequality for modified quermassintegrals  of  smooth strictly h-convex bodies. More precisely, they obtained that for any smooth strictly h-convex body $\Omega$ there holds
$$\tilde{W}_k(\Omega) \geq \tilde{f}_k \circ \tilde{f}_l^{-1}\left(\tilde{W}_l(\Omega)\right), \quad 0\leq l<k \leq n,
$$
and equality holds if and only if $\Omega$ is a geodesic ball. Here the function $\tilde{f}_i: [0,\infty) \rightarrow \mathbb{R}_{+}$ is defined by $\tilde{f}_i(r)=\tilde{W}_i(B(r))=\sum_{j=0}^i(-1)^{i-j}\binom{i}{j} f_j(r)$. The inequality above is called the modified Alexandrov-Fenchel inequality.

Smooth strictly h-convex approximations for h-convex bodies and the continuity extend the inequality statements above to arbitrary h-convex bodies. But the equality rigidities can not be obtained directly by the smooth approximation. In this paper, we characterize the equality conditions in the Alexandrov-Fenchel inequality and the modified Alexandrov-Fenchel inequality.
\begin{thm}\label{A-F}
For any  h-convex body $\Omega$ in $\mathbb{H}^{n+1}$ and $0\leq l<k \leq n$, we have that
\begin{align*}
W_k(\Omega) &\geq f_k \circ f_l^{-1}\left(W_l(\Omega)\right),\\
\tilde{W}_k(\Omega) &\geq \tilde{f}_k \circ \tilde{f}_l^{-1}\left(\tilde{W}_l(\Omega)\right),
\end{align*}
and equality in each inequality holds if and only if $\Omega$ is a geodesic ball.
\end{thm}

We note that the equality rigidities in the smooth case follow by classical techniques. If we consider smooth h-convex bodies, the equality rigidity of the Alexandrov-Fenchel inequality is directly obtained by considering the $1$-homogeneous Quermassintegral preserving curvature flow introduced by Wang and Xia \cite{WangXia2014}.  For the modified Alexandrov-Fenchel inequality, noticing that any bounded smooth domain in hyperbolic space has at least one point at the boundary at which the principal curvatures are strictly greater than $1$, one can consider the local variation around this point, which is similar to arguments in \cite[Theorem 2]{GuanLi2009}, to obtain the equality rigidity of the modified Alexandrov-Fenchel inequality for smooth h-convex bodies. The main difficulty is to establish the equality rigidities for  h-convex bodies which may not be smooth.

We adapt different arguments to obtain the equality conditions above. For the Alexandrov-Fenchel inequality, we consider the Hausdorff stability, which is independent of the curvature bounds, of the Alexandrov-Fenchel inequality described by the deficit $W_k(\Omega)-f_k\circ f_{l}^{-1}(W_l(\Omega))$. For the  modified Alexandrov-Fenchel inequality, we characterize the geodesic balls among h-convex bodies by the proportionality between modified support measures of adjacent orders and extend the modified quermassintegrals preserving flows \eqref{smoothflow1} to arbitrary initial h-convex bodies. We first introduce the rigidity theorem for modified support measures.

Curvature measures and area measures extend the concept of  curvature quantities of smooth convex hypersurfaces to  convex bodies without the smooth assumption in Euclidean space $\mathbb{R}^{n+1}$. Area measures were introduced by Alexandrov \cite{Alexandrov1937} and by Fenchel and Jessen \cite{FenchelJenssen}. And  Federer \cite{Federer1959} introduced the curvature measures for sets of positive reach in Euclidean space. Then Schneider \cite{Schneider1979} linked concepts of curvature measures and area measures together through the concept of support measures. The support measures are concentrated on the unit normal bundles and weakly continuous with respect to the  Hausdorff metric of convex bodies, see \cite[Chapter 4]{Schneider}.  Kohlmann \cite{Kohlmann1991} extended the concept of the curvature measures to  sets of positive reach in space forms.  By the calculations in \cite{Kohlmann1991}, one can obtain the support measures defined on the normal bundles for convex bodies in space forms. 

In this paper, we focus on  h-convex bodies in the hyperbolic space $\mathbb{H}^{n+1}$.  For a smooth strictly h-convex body $\Omega$, Andrews, Chen and Wei \cite{AndrewsChenWei2021} developed  the horospherical Gauss map, which maps each point $x$ on $\partial \Omega$ to the ideal point related to the direction opposite to that of the outer normal vector of $\partial\Omega$ at $x$. The ideal boundary $\partial _\infty\mathbb{H}^{n+1}$ is identified with $\mathbb{S}^n$. Moreover, they determined the smooth  strictly h-convex body by its horospherical support functions and introduced a corresponding positive $(0,2)$ tensor $A(\varphi_\Omega)$, which is related to the inverse of the shifted Weingarten map $W-I$. See also \cite{Espinar09} for an earlier related construction of the horospherical Gauss map with different settings. Using this description, Li and Xu \cite{LiXu2022} introduced horospherical $p$- surface area measures of smooth strictly h-convex bodies through the hyperbolic $p$-sum structure. The case $p=0$ corresponds to the modified area measures we considered here. The author \cite{Zhang2026} introduced a concept of modified curvature measures for h-convex bodies in hyperbolic space.  Li and Xu \cite{LiXu2022} derived a Steiner type  formula for modified quermassintegrals of  smooth strictly h-convex bodies. Combining their calculations with the local Steiner formula for sets of positive reach in \cite{Kohlmann1991},  one can derive a modified version of the Steiner formula for arbitrary h-convex body. The modified area measures and modified curvature measures are actually the coefficients in the modified Steiner formula. The natural projection of the unit normal bundle and a projection to the ideal boundary recovered the modified curvature measures and the modified area measures as the modified support measures, respectively. 

There are several works characterizing convex bodies via area measures and curvature measures. For area measures,  the determination, up to translation, of a convex body with nonempty interior by the  area measures was introduced by Alexandrov \cite{Alexandrov1937II} and  Fenchel and Jessen \cite{FenchelJenssen} independently.  Berg \cite{Berg1969} used spherical potential theory to relate  the support function of a convex body to its first order area measure and obtained the characterization of the first area measure of a convex body. In particular, since the $0$-th order area measure of any convex body is a Lebesgue measure on $\mathbb{S}^n$, the proportionality of the $0$-th order area measure and first order area measure characterizes the geodesic ball. For curvature measures,  Schneider \cite{Schneider1979} characterized the geodesic balls via curvature measures, which is a result generalizing the classical Alexandrov theorem. 
Kohlmann subsequently established the generalized Minkowski integral formulas for compact convex bodies \cite{Kohlmann94} and characterized the geodesic balls via curvature measures \cite{Kohlmann1998} in space forms.
In the author's earlier work \cite{Zhang2026}, the author proved that, for any h-convex body $\Omega$, if 
\begin{align*}
    \tilde{V}_{k}(\Omega,\cdot)=c\tilde{V}_n(\Omega,\cdot), \quad c>0,\quad 0\leq k\leq n-2,
\end{align*}
then $\Omega$ is a geodesic ball.

Let $\tilde{S}_k$, $\tilde{V}_k$ and $\tilde{\Phi}_k$ be the $k$-th modified  area measures, the $k$-th  modified curvature measures and the $k$-th modified support measures for h-convex bodies. In Section \ref{preli} we will introduce the concepts in more detail. In this paper, we prove that the proportionality between modified support measures of adjacent orders characterizes geodesic balls.

 \begin{thm}\label{metricball}
    Suppose that $\Omega$ is an h-convex body and $0\leq k\leq n-1$. If there exists a constant $c>0$ such that  
    \begin{align}
     \tilde\Phi_{n-k-1}(\Omega,\cdot)=c\tilde{\Phi}_{n-k}(\Omega,\cdot),\label{conditionsofrigisupp} 
    \end{align}
    then $\Omega$ is a geodesic ball.
\end{thm}
 We note that, in the rigidity results for curvature measures above, the proportionality conditions involve the $n$-th order curvature measures. Actually, by \cite[Theorem 5.14 (4)]{Federer1959}, the proportionality between curvature measures of adjacent orders without $n$-th order curvature can not characterize geodesic balls in Euclidean space $\mathbb{R}^{n+1}$. And Schneider \cite{Schneider1978} pointed out that the proportionality of area measures of two adjacent orders without $0$-th order area measure can not characterize geodesic balls. Since the support measures contain more geometric information, we consider the Theorem \ref{metricball}.

In the following we reduce \eqref{conditionsofrigisupp} to proportionality between the support measures of the  lowest adjacent  orders. 
By the ideal boundary projection, for an h-convex body $\Omega$, if $\tilde{\Phi}_0(\Omega,\cdot)=c\tilde{\Phi}_1(\Omega,\cdot),$ we have $$\tilde{S}_0(\Omega,\cdot)=c\tilde{S}_1(\Omega,\cdot).$$ 
 The following Theorem is therefore a key ingredient.
 \begin{thm}\label{arearigidity}
    Suppose that $\Omega$ is an h-convex body. If there exists a constant $c>0$ such that  
    \begin{align}\label{areatiaojian}
     \tilde S_{0}(\Omega,\cdot)=c\tilde{S}_{1}(\Omega,\cdot),
     \end{align}
    then $\Omega$ is a geodesic ball.
\end{thm}

In the smooth strictly h-convex case, the prescribed curvature problem of modified area measures has  relations with the   Nirenberg-Kazdan-Warner problem, which were mentioned in \cites{Espinar09,LuoWei2025,LiXu2022}. 
If $\Omega$ is smooth strictly h-convex, Theorem \ref{arearigidity} follows in \cite[Theorem 8.1 (7)]{LiXu2022}.  
 
In this paper, we associate each h-convex body in hyperbolic space $\mathbb{H}^{n+1}(n\geq 2)$ with a nonnegative Radon measure valued $(0,2)$-tensor $A[\varphi_\Omega]$ on the sphere $\mathbb{S}^n$ through its horospherical function. Using spherical convolution, we prove that any h-convex body can be approximated by a sequence of smooth strictly h-convex bodies in the Hausdorff sense. 
\begin{lemma}\label{smoothapproxiamate}
    For any h-convex body $\Omega\subset \mathbb{H}^{n+1}$, there exists a sequence of smooth strictly h-convex bodies $\left\{\Omega_i\right\}_{i\in \mathbb{N}}$ such that $\Omega\subseteq\Omega_{i+1}\subseteq\Omega_{i}(\Omega_{i}\subseteq\Omega_{i+1}\subseteq\Omega)$ and $$\Omega_i\rightarrow \Omega\, \text{in the Hausdorff sense},\,\, \text{as}\, i\rightarrow \infty.$$
\end{lemma}

 By the smooth approximation lemma above together with calculations of Li and Xu \cite{LiXu2022}, we derive representations of the $0$-th order area measure and the first order area measure, in the spirit of Berg's work \cite{Berg1969}. Then the Theorem \ref{arearigidity} is actually a measure valued equation on $\mathbb{S}^n$. 
 
 To reduce the Theorem \ref{metricball} to the Theorem \ref{arearigidity}, a difficulty in using the modified generalized Minkowski formulas \cite{Zhang2026} and the Newton-Maclaurin inequality is that the generalized principal curvatures of nonsmooth h-convex bodies may equal  $\infty$. To overcome this difficulty, under the assumption \eqref{conditionsofrigisupp}, we consider the outer parallel sets and approximate their $C^{1,1}$ boundaries by smooth strictly h-convex hypersurfaces in the sense of $C^{1,\alpha}\cap W^{2,k+1}$, by which we can construct a sequence of smooth strictly h-convex approximation hypersurfaces whose total variation norms of the signed measure associated with the proportionality relation tend to $0$. By argument above, we express the measure relation \eqref{conditionsofrigisupp} on the nonsmooth h-convex body by the curvature condition of smooth strictly h-convex hypersurfaces. 
 
Then we evolve these smooth hypersurfaces along the shifted curvature flow \eqref{smoothflow1} and establish the time dependent estimates on the curvature quotients $\frac{\tilde{E}_{k+1}}{\tilde{E}_k}$ and $\mathcal{L}^1$ norm of $D=\tilde{E}_{k+1}-\phi(t)\tilde{E}_k$ by the initial $\mathcal{L}^1$ norms, uniformly along these hypersurfaces and independently of the curvature bounds. The uniform estimates above, which are independent of initial curvature bounds, allow us to estimate the $\mathcal{L}^1$ norms of $D$ on the evolving approximating hypersurfaces at any fixed positive time. Then the uniform bounds on the curvature quotients $\frac{\tilde{E}_{k+1}}{\tilde{E}_k}$ allow us to apply the modified generalized Minkowski formulas and the Newton-Maclaurin inequality to choose a sequence of smooth h-convex approximation hypersurfaces to reduce the proportionality relation \eqref{conditionsofrigisupp} to that of lower order. The argument above allows us to express the information of  modified support measures of an h-convex body by the curvature quantities on a sequence of  suitable smooth h-convex hypersurfaces which approximate its boundary in the Hausdorff sense.

We note that the same approach above also yields analogous rigidity results for the support measures of convex bodies in Euclidean space and h-convex bodies in hyperbolic space. One can use the global constrained flows introduced by McCoy \cite{McCoy2005} and  Wang and Xia \cite{WangXia2014} to obtain the rigidity results for support measures in Euclidean space and hyperbolic space, respectively.

The curvature flows with nonsmooth initial hypersurfaces were studied in several works. 
Bellettini, Caselles, Chambolle and Novaga \cite{BCCN2009} considered the volume preserving crystalline mean curvature flows. Moreover, they considered the stability with respect to the Hausdorff distance of initial hypersurfaces for volume preserving anisotropic mean curvature flows. Andrews, McCoy and Zheng \cite{Andrews2013} considered the contracting curvature flows for general compact convex hypersurfaces in Euclidean space. Choi and Hung \cite{ChoiHung2023} considered the inverse mean curvature flows for general compact convex hypersurfaces in Euclidean space.

In this paper, we consider the modified quermassintegral preserving flows with general h-convex initial hypersurfaces. For any h-convex body $\Omega$, by Lemma \ref{smoothapproxiamate}, there is a sequence of smooth strictly h-convex bodies $\Omega_j$ such that $\Omega_j\rightarrow \Omega$ in the Hausdorff sense. Consider a sequence of flows
\begin{align*}
\left\{\begin{aligned} 
\frac{\partial}{\partial t} X_j(x, t) & =\left(\phi(t)-\frac{\tilde{E}_{k+1}}{\tilde{E}_k}(W)\right) \nu(x, t), \\ 
X_j(\cdot,0) & =X_{j,0}(\cdot),\end{aligned}\right.
\end{align*}
where $X_{j,0}:N^n\rightarrow M_j=\partial \Omega_j$, $\nu(x, t)$ is the unit outer normal of the hypersurface $M_{j,t}:=X_j(N^n,t)=\partial \Omega_{j,t}$, $W$ is the Weingarten matrix of $M_{j,t}$, $\tilde{E}_k(x,t):=E_k(\mathcal{W}-I)$ is the shifted  $k$-th mean curvature of $M_{j,t}$ and  \begin{align}
    \phi(t):=\frac{\int_{M_{j,t}}\tilde{E}_{k+1}(\mathcal{W}) \mathrm{d} \mu }{\int_{M_{j,t}} \tilde{E}_k(W)d\mu} \end{align} is the global term chosen to preserve the $k$-th modified quermassintegral of the enclosed domains. We define the flow \begin{align}
    M_t=\partial\Omega_t:=\lim_{j\rightarrow\infty}M_{j,t}, t\geq 0.\label{flow}
    \end{align}   
We prove the following theorem. 
\begin{thm}\label{weakflowthm}
The flow \eqref{flow} is well defined, exists for all time and converges to a geodesic ball. For any $0\leq t<s$, we have $$d_{\mathbb{H}^{n+1}}(\Omega_t,\Omega_s)\leq C(\rho_-(\Omega),n,k)(s-t+\sqrt{s-t})$$. The flow \eqref{flow} preserves h-convexity and the $k$-th modified Quermassintegral, and  the $(k+1)$-th modified Quermassintegral is non-increasing.  
Moreover, the $(k+1)$-th modified Quermassintegral strictly decreases in a short time if $\Omega$ is not a geodesic ball.
\end{thm}

Theorem \ref{metricball} is used to obtain the strict monotonicity of $\tilde{W}_{k+1}$ along the flow \eqref{flow} when the initial h-convex body is not a geodesic ball and to characterize the limit shape. 

There are several works that apply curvature measures to characterize the limit shape of curvature flows, see \cites{ALWX25,AndrewsWei2021, WeiYangZhou2024, WeiYangZhou2025, Zhang2026}. More precisely, they derived the curvature measure identities for limit shapes of the flows by analyzing the $\mathcal{L}^2$ or $\mathcal{L}^1$ oscillation decay of curvatures.  In this paper, once we have established the rigidity theorem for modified  support measures, by the weak continuity of the modified support measures with respect to Hausdorff distance of h-convex bodies, this rigidity theorem gives a lower bound on the rate of decrease of $\tilde{W}_{k+1}(\Omega_t)$ when $\Omega$ is not a geodesic ball.
For any h-convex body $\Omega$ and $0\leq l\leq k\leq n-1$, the equality  $$\tilde W_{k+1}(\Omega)=\tilde{f}_{k+1}\circ\tilde{f}_l^{-1}(\tilde{W}_l(\Omega))$$ implies $$\tilde W_{k+1}(\Omega)=\tilde{f}_{k+1}\circ\tilde{f}_k^{-1}(\tilde{W}_k(\Omega)).$$
 Hence, by Theorem \ref{weakflowthm}, we immediately obtain the equality rigidity of the modified Alexandrov-Fenchel inequality for h-convex bodies. 
 
For the convergence of the flow \eqref{flow}, as in \cite{ALWX25,AndrewsWei2021, WeiYangZhou2024, WeiYangZhou2025, Zhang2026}, we first consider the subsequential convergence and apply Theorem \ref{metricball} to characterize the limit shape.  Then applying the  Alexandrov reflection arguments in \cite{AndrewsChenWei2021, WeiYangZhou2024, WeiYangZhou2025} we obtain the convergence. 

For the Alexandrov-Fenchel inequality, we establish the Hausdorff stability independent of curvature bounds. 
 The Fraenkel asymmetry stability described by the deficit  $W_1(\Omega)-f_1\circ f_{0}^{-1}(W_0(\Omega))$ in hyperbolic space was established by    B\"ogelein, Duzaar and Scheven \cite{Verena2015} for sets of finite perimeter. Moreover, they obtained the isoperimetric inequality in hyperbolic space.  There are several works on the stability described by the higher order Alexandrov-Fenchel deficit. Sahjwani and Scheuer \cite{SahjwaniScheuer2024} obtained the stability of the Alexandrov-Fenchel deficit for smooth strictly h-convex domains by the curvature flows. Moreover, Gao and Zhou \cite{Gao2024} obtained the stability of the modified Alexandrov-Fenchel deficit. The stability results in \cite{SahjwaniScheuer2024, Gao2024} need the information of the curvature quotients, which depend on the curvature bound. In our case, the boundary of a nonsmooth h-convex domain may have singular points, for which the stability results of Sahjwani and Scheuer \cite{SahjwaniScheuer2024} above cannot imply our result directly.  
Let $D_{k,l}(\Omega)=W_k(\Omega)-f_k(f^{-1}_l(W_l(\Omega)))$ and $B(\Omega)=\{B: B\,\, is\,\, a\,\, metric \,\,ball \,\,and \,\,W_0(B)=W_0(\Omega)\}$. As in \cite{Verena2015}, denote $$\alpha(\Omega)=\min\limits_{B\in B(\Omega)}\frac{|\Omega\Delta B|}{|B|}.$$
We have the following theorem.
\begin{thm}\label{stability}
    Suppose that $\Omega\in\mathcal{H}(\mathbb{H}^{n+1})$ and $\frac{\rho}{2}\leq \mathrm{Vol}(\Omega)\leq \rho$. Then we have 
    $$\alpha^2(\Omega)\leq C(n,k,l,\rho)D_{k,l}(\Omega), 0\leq l<k\leq n.$$
\end{thm}

As mentioned in \cite[Remark 1.10]{AndrewsChenWei2021}, the modified quermassintegrals give more precise details of h-convex geometry. By the inequality parts of the Alexandrov-Fenchel inequality and the modified Alexandrov-Fenchel inequality, we analyze $D_{k,l}$ in more detail. More precisely, by the modified Alexandrov-Fenchel inequality we can obtain the monotonicity of $\tilde{f}^{-1}_i(\tilde{W}_i(\Omega))$ with respect to $i$, from which we can  prove that $D_{1,0}(\Omega)\leq C(n,k,l,\rho)D_{k,l}(\Omega)$.

The theorem above implies the Hausdorff stability we need. More precisely, suppose that $\Omega$ is an h-convex body. We next apply the result in \cite{Groemer2000}, which controls the Hausdorff distance by the volume of their symmetric difference. Consider the Poincare ball model of $\mathbb{H}^{n+1}$ and take the center of the inner ball of $\Omega$ as the origin. The h-convexity of $\Omega$ under the hyperbolic metric  induces the convexity of $\Omega$ under the Euclidean metric.
Noticing that the hyperbolic metric and Euclidean metric can be controlled by each other in a compact domain, by the result in \cite{Groemer2000} we can control the Hausdorff distance between $\Omega$ and $B(\Omega)$ by $\alpha(\Omega)$, from which we obtain the following Hausdorff stability result.
\begin{thm}\label{hausdorffstability}
    Suppose that $\Omega\in\mathcal{H}(\mathbb{H}^{n+1})$ and $\frac{\rho}{2}\leq \mathrm{Vol}(\Omega)\leq \rho$. Then we have 
    $$\min\limits_{B\in B(\Omega)}\{d_{\mathbb{H}^{n+1}}(\Omega, B)\}\leq C(n,k,l,\rho)D^\frac{1}{2n+2}_{k,l}(\Omega), 0\leq l<k\leq n.$$
\end{thm}

 \begin{rmk} 
   The smooth approximation Lemma \ref{smoothapproxiamate}  was also obtained in the third version of Chen and Li's preprint \cite{ChenLi2026}. Their construction uses log-sum-exp smoothing of Busemann functions of finitely many horoballs. Our approach is based on the relation between the horospherical support functions and the nonnegative Radon measure valued tensors. This approach also provides some properties of horospherical support functions and more details on the modified area measure. The proof of Lemma \ref{smoothapproxiamate} was completed before the author became aware of the approximation result in the third version of Chen and Li's preprint \cite{ChenLi2026}.
\end{rmk}

\subsection{Organization of this paper}
In Section \ref{preli} we introduce some preliminaries including the quermassintegrals, modified quermassintegrals, generalized principal curvatures, modified support measures and some geometric properties of h-convex bodies and related shifted curvature flows. Moreover, we introduce the convolution on $\mathbb{S}^n$. In Section \ref{smoothsection}, we prove Lemma \ref{smoothapproxiamate} and introduce some properties of the horospherical support function.  In Section \ref{stabilityforaf}, we derive the Fraenkel asymmetry stability for the Alexandrov-Fenchel inequality. In Section \ref{flowsection}, we give some calculations of the smooth curvature flows. In Section \ref{chraofgeoball}, we characterize geodesic balls via modified support measures and modified area measures. In Section \ref{af}, we prove Theorem \ref{weakflowthm}, which implies the modified Alexandrov-Fenchel inequality. 

\subsection{On the use of AI}
The main ideas and guiding proof strategy of this paper were inspired by the author's earlier work \cite{Zhang2026} and were further developed by the author in the present work. The main theorems in this paper were formulated by the author. The author used ChatGPT to assist with literature searches, checking and refining mathematical details, and to suggest improvements to the exposition and language. The author personally carried out all calculations included in the manuscript. The author takes full responsibility for all results, proofs, references and exposition in this work.

\section{Preliminaries}\label{preli}
  For any smooth closed hypersurface $M\subseteq\mathbb{H}^{n+1}$, denote by $g=\{g_{ij}\}$, $\mathcal{W}=\{h^j_i\}$ ,$\nabla$, $\overline{\nabla}$ the induced metric and Weingarten matrix, induced Levi-Civita connection of $M$ and  Levi-Civita connection of $\mathbb{H}^{n+1}$ respectively. We denote by $B(x,r)$ the geodesic ball centered at $x$ of radius $r$ in $\mathbb{H}^{n+1}$.
Denote by $\mathcal{H}(\mathbb{H}^{n+1})$ the collection of h-convex bodies in $\mathbb{H}^{n+1}$. Equip the tangent bundle $T\mathbb{H}^{n+1}$ with the Sasaki metric with respect to the metric on $\mathbb{H}^{n+1}$ denoted by $\langle \cdot, \cdot\rangle_{\mathbb{H}^{n+1}}$. Denote by $\Pi :T\mathbb{H}^{n+1}\rightarrow \mathbb{H}^{n+1}$ the canonical projection and by $\Gamma:TT\mathbb{H}^{n+1}\rightarrow T\mathbb{H}^{n+1}$ the connection map. For any $w\in T\mathbb{H}^{n+1}$ we identify $T_wT\mathbb{H}^{n+1}$ with $T_{\Pi(w)}\mathbb{H}^{n+1}\times T_{\Pi(w)}\mathbb{H}^{n+1}$ by the isometry $\Pi_*\times\Gamma$.
\subsection{Quermassintegrals for convex bodies in hyperbolic space}
In hyperbolic space, for any convex body $\Omega$ in $\mathbb{H}^{n+1}$, the quermassintegrals of $\Omega$ are defined as (see \cites{integralgeometrybook, integralgauss}) 
$$W_k(\Omega)=\frac{\omega_{k-1} \cdots \omega_0}{\omega_{n-1} \cdots \omega_{n-k}} \int_{\mathcal{L}_k} \chi\left(L_k \cap \Omega\right) d L_k, \quad k=1, \ldots, n,$$
where $\mathcal{L}_k$ is the space of $k$-dimensional totally geodesic subspaces in $\mathbb{H}^{n+1}$, $\omega_n$ denotes the area of the $n$-dimensional unit sphere in Euclidean space, and $\chi$ is the Euler characteristic.
 Furthermore, set $W_0(\Omega)=|\Omega|$, $W_{n+1}(\Omega)=\left|\mathbb{B}^{n+1}\right|=\frac{\omega_n}{n+1} .$
\subsection{Hyperbolic geometry and horospherical support functions}
In this section we consider the Minkowski space $(\mathbb{R}^{n+1,1},\langle,\rangle_L)$ and  identify the Hyperbolic space $$\mathbb{H}^{n+1}=\left\{ X\in\mathbb{R}^{n+1,1}: \langle X,X\rangle_L=-1\right\}$$ with the future time-like hyperboloid.
  Denote by $\mathcal{C}(\mathbb{H}^{n+1})$ the collection of compact sets in $\mathbb{H}^{n+1}$. It is known that a compact domain $\Omega$ is h-convex iff any one of the following conditions holds:

(1). For any $x\in\partial \Omega$, there exists a supporting horoball of $\Omega$ at $x$.

(2). For any $x,y\in \Omega$($x\neq y$), every horocycle segment from $x$ to $y$ lies in $\Omega$.

For any $\varepsilon>0$,   the outer and inner parallel sets of $\Omega\subset \mathbb{H}^{n+1}$ are defined by
$$\Omega^{\varepsilon}=\left\{x \in \mathbb{H}^{n+1} : d_{\mathbb{H}^{n+1}}(x, \Omega) \leq \varepsilon\right\},$$ 
$$\Omega_{\varepsilon}=\left\{x \in \mathbb{H}^{n+1} :x\in\Omega, d_{\mathbb{H}^{n+1}}(x, \partial\Omega) \geq \varepsilon\right\}.$$ 
In \cite{Zhang2026}, the author proved that any outer parallel set of an h-convex body is h-convex. 
\begin{lemma}[\cite{Zhang2026}]\label{outerhconvex}
    Suppose that $\Omega\in\mathcal{H}(\mathbb{H}^{n+1})$. Then for any $\varepsilon>0$, we have $\Omega^\varepsilon\in \mathcal{H}(\mathbb{H}^{n+1})$. 
\end{lemma}
Moreover, the author proved a Blaschke selection type theorem for h-convex bodies.
\begin{thm}[\cite{Zhang2026}]
   For any bounded sequence $\{\Omega_i, i\in\mathbb{N}\}\subset \mathcal{H}(\mathbb{H}^{n+1})$ whose elements  have a uniform positive lower bound on their inner radius, there is a subsequence which converges to $\Omega\in \mathcal{H}(\mathbb{H}^{n+1})$.\label{hausdorffofhconvex}
\end{thm}

For any $\xi\in\mathbb{H}^{n+1}$, define $f_\xi(e)=-\langle\xi,(e,1)\rangle$ for any $e\in\mathbb S^{n}$. 
Following the recent work of Andrews, Chen and Wei \cite{AndrewsChenWei2021}, for any compact domain $\Omega\subset \mathbb{H}^{n+1}$, define the horospherical support function $u_\Omega: \mathbb{S}^n\rightarrow R$ by 
   $$u_\Omega(e):=\sup\limits_{\xi\in \Omega} \{\log f_\xi(e)\}.$$
Moreover, they defined $\varphi_\Omega=e^{u_\Omega}>0$ and established a one-to-one correspondence between a smooth strictly h-convex body $\Omega$ and a smooth function $u_\Omega$ satisfying $$A(\varphi_\Omega):=\overline{\nabla}^2 \varphi_\Omega+\left(\frac{\varphi_\Omega-\varphi_\Omega^{-1}}{2}-\frac{|\overline{\nabla}\varphi_\Omega|^2}{2\varphi_\Omega}\right)\overline{g}>0,$$
 and proved that $$A(\varphi_\Omega)_{ij}=e^{-u_\Omega}[(\mathcal{W}-I)^{-1}]^k_i\overline{g}_{kj},$$ where $\mathcal{W}$ is the Weingarten matrix of $\partial \Omega$ at the point $x\in\partial\Omega$ satisfying $u_\Omega(e)=\log f_x(e)$.
When an h-convex domain degenerates to a point, Li and Xu \cite{LiXu2022} proved that 
\begin{align}\label{singlepoint}
    A(\varphi_{\{\xi\}})=0.
\end{align}
 For any $\Omega\in\mathcal{H}(\mathbb{H}^{n+1})$, the inner radius $\rho_{-}(\Omega)$ and the outer radius $\rho_{+}(\Omega)$ of $\Omega$ are defined by $$\rho_{-}(\Omega)=\sup \left\{\rho: B_\rho(p) \subset \Omega \text { for some } p \in \mathbb{H}^{n+1}\right\}$$ and $$\rho_{+}(\Omega)=\inf \left\{\rho: \Omega \subset B_\rho(p) \text { for some } p \in \mathbb{H}^{n+1}\right\}.$$
 It is known that for any $\Omega\in\mathcal{H}(\mathbb{H}^{n+1})$ the inner radius $\rho_{-}(\Omega)$ and the outer radius $\rho_{+}(\Omega)$ can be compared as follows:

\begin{thm}[\cites{BorisenkoGallegoReventos2001, BorisenkoMiquel1999,BorisenkoMiquel2002,Makowski2012,CabezasMiquel2007}]\label{hconvex}
    For any $\Omega\in\mathcal{H}(\mathbb{H}^{n+1})$, let $o$ be the center of an inner ball of $\Omega$. We have

    (1) The maximum distance from $o$ to $\partial \Omega$ is controlled by $\rho_-(\Omega)$
    \begin{align}
        \max_{p \in \partial \Omega} d_{\mathbb{H}^{n+1}}(o, p) \leq \rho_{-}(\Omega)+\ln \frac{\left(1+\sqrt{\tanh \frac{\rho_{-}(\Omega)}{2}}\right)^2}{1+\tanh \frac{\rho_{-}(\Omega)}{2}}<\rho_{-}(\Omega)+\ln2.\label{innerandouter}
    \end{align}
    Moreover, there exists a constant $c>0$ such that $$\rho_{+}(\Omega) \leq c\left(\rho_{-}(\Omega)+\rho_{-}^{1/2}(\Omega)\right)$$.

    (2)  For any $p\in \mathring{\Omega}$, we have \begin{align}
        \langle \partial d_{p},\nu\rangle_{\mathbb{H}^{n+1}} \geq \tanh \left(\mathrm{d}_{\mathbb{H}}(p, \partial \Omega)\right), \text{for any $q\in\partial \Omega$}.\label{support}
    \end{align}
\end{thm}

\subsection{Generalized principal curvatures}\label{plementaryofgpc}
In this subsection we recall some basic concepts introduced by Kohlmann \cites{Kohlmann1991,Kohlmann94,Kohlmann1998}.
For any convex body $\Omega\subseteq\mathbb{H}^{n+1}$, we define $F_{\Omega}: \mathbb{H}^{n+1} \setminus \Omega \rightarrow T \mathbb{H}^{n+1}$ as
\[ x=\exp_{f_{\Omega}(x)}^{\mathbb{H}^{n+1}}\left(d_{\mathbb{H}^{n+1}}(\Omega, x) F_{\Omega}(x)\right),
\]
where $d_{\mathbb{H}^{n+1}}(x,y)$ denotes  the distance between $x,y\in\mathbb{H}^{n+1}$ and $f_\Omega$ is the metric projection onto $\Omega$.
Denote by $\mathcal{N}_\Omega$ the image of $F_\Omega$. By  \cite{Walter74, Walter76},  $\mathcal{N}_\Omega$ is a strong Lipschitz submanifold of the tangent bundle $T\mathbb{H}^{n+1}$ and $F_\Omega|_{\partial \Omega^t}$ is bi-Lipschitz for any $t>0$.

Kohlmann \cite{Kohlmann1991} proved that for almost all $v\in\mathcal{N}_\Omega$ one can define the generalized principal curvatures  $k_i(v)\in [0,\infty]$. Denote by $\tilde{\mathcal{N}}_{\Omega}\subseteq \mathcal{N}_\Omega$ the set on which $\kappa_i(1\leq i\leq n)$ are well defined. 
In the following we assume that $\kappa_1(v)\geq \kappa_2(v)\geq\dots\geq \kappa_n(v)$ for any $v\in \tilde{\mathcal{N}}_\Omega$ such that the $\kappa_i(v)$ exist. For $0\leq r \leq n$, let 
\[\tilde{\mathcal{N}}_{\Omega, r}:=\left\{v \in \tilde{\mathcal{N}}_\Omega \mid \kappa_1(v)=\ldots=\kappa_r(v)=\infty \text { and } \kappa_{r+1}(v)<\infty\right\} .\]
 For any $ v \in \tilde{\mathcal{N}}_{\Omega, r}$, we define $g_\Omega(v)$ as $$g_\Omega(v):=\infty^{-r} \prod_{i=r+1}^n\left(1+\kappa_i(v)^2\right)^{-1/2},$$
where "$\infty^{-r}$" is an algebraic expression. For any $v\in\tilde{N}_\Omega$, define the $k$-th generalized mean curvature  
$$E_k(v)=\binom{n}{k}^{-1}\sum\limits_{\substack{1\leq i_{1}<\cdots<i_k\leq n}}\kappa_{i_{1}}(v)\cdots\kappa_{i_k}(v).$$
In this definition, for any $v\in\tilde{N}_\Omega$ we have $g_\Omega(v)E_k(v)$ is a finite nonnegative real number.
 In \cite{Zhang2026} the author defined a concept of the shifted $k$-th mean curvature for $\Omega\in\mathcal{H}(\mathbb{H}^{n+1})$ by 
$$\tilde E_k(v)=\binom{n}{k}^{-1}\sum\limits_{\substack{1\leq i_{1}<\cdots<i_k\leq n}}(\kappa_{i_{1}}(v)-1)\cdots(\kappa_{i_k}(v)-1)$$ and proved 
\begin{align}
    g_\Omega(v)\tilde{E}_k(v)=\sum\limits_{\ell=0}\limits^{k}(-1)^{k-\ell}\binom{k}{\ell}g_\Omega(v)E_\ell(v), \,\text{for any } v\in\tilde{\mathcal{N}}_{\Omega}.\label{relationsbetweenge}
\end{align}
Moreover, the author \cite{Zhang2026} proved that for any $\Omega\in\mathcal{H}(\mathbb{H}^{n+1})$, there holds $k_i(v)\geq 1, \forall v\in\tilde{N}_\Omega$.

As is remarked by Kohlmann \cite{Kohlmann1991}, for any $t>0$, $M_t=\partial\Omega^t$ is of class $C^{1,1}$. Moreover, for any $v\in\tilde{\mathcal{N}}_{\Omega}$, the second fundamental form of $M_t$ is well defined at $\exp^{\mathbb{H}^{n+1}}_{\Pi(v)}(tv)\in M_t$ and the principal curvatures are related to $k_i(v)$ as follows
\begin{align}\label{waipingxingqvlv}
    k_i(\exp^{\mathbb{H}^{n+1}}_{\Pi(v)}(tv))=\frac{\cosh t k_i(v)+\sinh t}{\sinh t k_i(v)+\cosh t}.
\end{align}

\subsection{Support measures in hyperbolic space} We denote by $\mathcal{K}(\mathbb{H}^{n+1})$ the collection of  compact, convex subsets of  $\mathbb{H}^{n+1}$ with nonempty interior.  For any $\Omega\in\mathcal{K}(\mathbb{H}^{n+1})$, by the argument in \cite{Kohlmann1991}  we know that   there exist Borel measures $\Phi_k(\Omega,\cdot)$ on $T^1\mathbb{H}^{n+1}$ such that for any Borel set $\beta\subseteq T^1\mathbb{H}^{n+1}$, there holds 
\begin{align}
\text{Vol}_{\mathbb{H}^{n+1}}(F_\Omega^{-1}(\beta)\cap(\Omega^\varepsilon\setminus\Omega))=\sum_{k=0}^{n} \ell_{n+1-k}(\varepsilon) \Phi_k(\Omega, \beta),\label{steiner}
\end{align}
where 
$$\ell_{k+1}(t)=\int_0^t \sinh ^k(s) \cosh ^{n-k}(s) d s, \quad k=0,1, \cdots, n.$$ 
The coefficients $\Phi_0(\Omega, \cdot), \cdots, \Phi_{n}(\Omega, \cdot)$  are called the support measures of the convex body $\Omega$ that can be represented as \begin{align*}&\Phi_k(\Omega, \beta):=\binom{n}{k}\int_{\beta \cap \mathcal{N}_\Omega} g_\Omega(v) {E}_{n-k}(v) d \mathcal{L}^n v \quad \text { for }\,0\leq k \leq n, \end{align*}
where $E_k(v):=\binom{n}{k}^{-1} \sum\limits_{1\leq i_1<\ldots<i_k \leq n} \kappa_{i_1}(v)\cdot\ldots\cdot \kappa_{i_k}(v)$ for $1\leq k\leq n$ and $E_0(v)=1$. 
When the boundary $\partial \Omega$ is smooth, $\Phi_k(\Omega, \cdot)$ can be expressed as
$$\Phi_k(\Omega, \beta)=\binom{n}{k}\int_{\partial \Omega \cap \Pi(\beta\cap\mathcal{N}_\Omega)} E_{n-k}(\kappa) d \mu.$$

The concept of support measures above induces a concept of curvature measures  $V_0(\Omega, \cdot)$, $\cdots, V_{n+1}(\Omega, \cdot)$ on $\mathbb{H}^{n+1}$ such that for any $\beta\in\mathcal{B}(\mathbb{H}^{n+1})$, $$V_{n-k}(\Omega,\beta)=\Phi_{n-k}(\Omega,\Pi^{-1}(\beta)\cap\mathcal{N}_\Omega).$$

As in \cite{Zhang2026}, we can define a concept of modified support measures as 
\begin{align*}&\tilde\Phi_k(\Omega, \beta):=\binom{n}{k}\int_{\beta \cap \mathcal{N}_\Omega} g_\Omega(v) \tilde{E}_{n-k}(v) d \mathcal{L}^n v \quad \text { for }\,0\leq k \leq n,\end{align*}
and $\tilde{E}_0(v)=1$. In the following, let $\tilde{\Phi}_k(\Omega)=\tilde{\Phi}_k(\Omega,\mathcal{N}_\Omega)$, $0\leq k\leq n$.

The concept of modified support measures above induces a concept of modified curvature measures  $\tilde V_0(\Omega, \cdot)$, $\cdots, \tilde V_{n+1}(\Omega, \cdot)$ on $\mathbb{H}^{n+1}$ such that for any $\beta\in\mathcal{B}(\mathbb{H}^{n+1})$, $$\tilde V_{n-k}(\Omega,\beta)=\tilde\Phi_{n-k}(\Omega,\Pi^{-1}(\beta)\cap\mathcal{N}_\Omega).$$

For any $v\in T^1\mathbb{H}^{n+1}$, let $\pi(v)$ be the ideal point $\xi\in\mathbb{S}^n$ corresponding to $-v$. We define the following horospherical Gauss map 
\begin{align*}
G_\Omega(p):=\{\pi(v):v\in\mathcal{N}_\Omega\cap T_p\mathbb{H}^{n+1}\}\subseteq \mathbb{S}^n,     
\end{align*}
which induces a concept of modified area measures for h-convex bodies: 
\begin{align*}
    \tilde{S}_{n-k}(\Omega,\beta):=\tilde{\Phi}_{n-k}(\Omega,\pi^{-1}\beta),\quad \beta\in\mathcal{B}(\mathbb{S}^n).
\end{align*}
When $\Omega$ is smooth strictly h-convex, by the calculations in \cite{LiXu2022}, we have 
\begin{align*}
 \tilde{S}_{n-k}(\Omega,\beta)=\int_{\mathbb{S}^n\cap \beta}  \varphi_\Omega^{-k} \sigma_{n-k}(A(\varphi_\Omega))d\mathcal{L}^n.  
\end{align*}

In \cite[Theorem 2.4]{Kohlmann94} Kohlmann  proved  the following result.         
\begin{thm}[\cite{Kohlmann94}]\label{weakcon}
    Let $\Omega, \Omega_1, \Omega_2, \dots$ be compact convex domains in $\mathbb{H}^{n+1}$ and \[\lim\limits_{i\rightarrow\infty}d_{\mathbb{H}^{n+1}}(\Omega,\Omega_i)=0.\] Then we have \begin{align*}
        V_k(\Omega_i,\cdot)\stackrel{w}\rightharpoonup V_k(\Omega,\cdot), \quad k=0,\dots,n.
    \end{align*}
\end{thm} 
   
As in Kohlmann's proof \cite{Kohlmann94}, we have the following theorem.
\begin{thm}
    \label{weakconsupport}
    Let $\Omega, \Omega_1, \Omega_2, \dots$ be compact convex domains in $\mathbb{H}^{n+1}$ and suppose that $$\lim\limits_{i\rightarrow\infty}d_{\mathbb{H}^{n+1}}(\Omega,\Omega_i)=0.$$ Then we have \begin{align*}
        \Phi_k(\Omega_i,\cdot)\stackrel{w}\rightharpoonup \Phi_k(\Omega,\cdot), \quad k=0,\dots,n.
    \end{align*}
\end{thm}
\begin{proof}
    For any $p\in \mathbb{H}^{n+1}\setminus \Omega$, Kohlmann \cite{Kohlmann94} proved that $\Pi(F_{\Omega_i})(p)\rightarrow\Pi(F_{\Omega})(p)$ as $i\rightarrow\infty$. Noticing that $F_{\Omega_i}(p)\in T^1{\mathbb{H}^{n+1}}$ is actually the unit vector such that $$\exp_{\Pi{F_{\Omega_i}(p)}}(d(\Pi(F_{\Omega_i}(p)),p)F_{\Omega_i}(p)),$$ we have $F_{\Omega_i}(p)\rightarrow F_{\Omega}(p)$ as $i\rightarrow\infty$. Then using similar arguments in \cite{Kohlmann94} we complete the proof.
\end{proof}

 For any $q\in \mathbb{H}^{n+1}$, let $$S_q(v):=\sinh{\left(d_q(\Pi(v))\right)}\left\langle v, \operatorname{grad} d_q(\Pi(v))\right\rangle_{\mathbb{H}^{n+1}}, v\in TM, ||v||=1.$$ 
  Kohlmann \cite[Theorem 3.4]{Kohlmann94} proved the generalized Minkowski integral formulas.
\begin{thm}[\cite{Kohlmann94}]\label{thmkohlmin}
    For any $\Omega \in \mathcal{H}(\mathbb{H}^{n+1})$ and $r \in\{1, \ldots, n\}$,  we have
$$\int_{\mathcal{N}_\Omega} g_\Omega(v)\left((\cosh d_q(\Pi(v)) E_{r-1}(v)-S_q(v) E_r(v)\right) \mathrm{d} \mathcal{L}^n v=0.$$
\end{thm}
As mentioned in \cite{Zhang2026}, the arguments in \cite[Lemma 2.6]{HuLiWei2022}, together with   Theorem \ref{thmkohlmin} and \eqref{relationsbetweenge} yield a modified  version of the  generalized Minkowski integral formulas for shifted principal curvatures.
\begin{thm}\label{shiftedminkow}
For any $\Omega \in \mathcal{H}(\mathbb{H}^{n+1})$ and $r \in\{1, \ldots, n\}$,  we have
$$\int_{\mathcal{N}_\Omega} g_\Omega(v)\left((\cosh\left(d_q(\Pi(v)))-S_q(v)\right) \tilde{E}_{r-1}(v)-S_q(v) \tilde{E}_r(v)\right) \mathrm{d} \mathcal{L}^n v=0.$$
\end{thm}

\subsection{Smooth modified quermassintegrals preserving shifted curvature flows}

Andrews-Chen-Wei \cite{AndrewsChenWei2021} considered the modified quermassintegrals preserving $1$-homogeneous shifted curvature flows with smooth strictly h-convex initial hypersurfaces. The following theorem is a direct conclusion of \cite[Theorem 1.7]{AndrewsChenWei2021}.
\begin{thm}[\cite{AndrewsChenWei2021}]\label{ACWsmoothflow}
Suppose that $M$ is a closed smooth strictly h-convex hypersurface. Then solutions to the flow \eqref{smoothflow1} with initial hypersurface $M$ preserve smooth strict h-convexity and the $k$-th modified quermassintegral, exist for all time and converge to a geodesic sphere.
\end{thm}
Let $S^i_j=h^i_j-\delta^i_j$ be the shifted Weingarten matrix. Denote by $F(S)=\frac{E_{k+1}(S)}{E_k(S)}$ the curvature function. In the following, denote $$\dot{F}^{ij}(S)=\frac{\partial F}{\partial S^{i}_j}.$$
We next calculate some evolution equations of the smooth flow \eqref{smoothflow1}. The following lemma is proved in \cite{AndrewsChenWei2021, HuLiWei2022}.
\begin{lemma}
    \begin{align}
     \frac{\partial}{\partial t}d\mu_t=(\phi(t)-F)H d\mu_t .\label{mianji} 
    \end{align}
    \begin{align}\label{Elevolution}
        \frac{\partial}{\partial t}\tilde{E}_l=\dot{\tilde{E}}_l^{ij}\nabla_i\nabla_j F+((F-\phi(t))((H-n+2l){\tilde{E}}_l-(n-l)\tilde{E}_{l+1}).
    \end{align}
    \begin{align}\label{weingartenevolution}
        \frac{\partial}{\partial t} S^{i}_j=&\dot{F}^{kl} \nabla_k\nabla_l S^i_j+\ddot{F}^{kl,pq} \nabla_i S_{kl} \nabla_jS_{pq} +(\dot{F}^{kl}(S^2)_{kl}+2F-2\phi(t))S^i_j\nonumber\\&-(\phi(t)+\dot{F}^{kl}\delta^l_k)(S^2)^i_j+\dot{F}^{kl}(S^2)_{kl}\delta^i_j.
    \end{align}
\end{lemma}
Andrews, Chen and Wei \cite{AndrewsChenWei2021} proved that the flow \eqref{smoothflow1} of h-convex hypersurfaces is equivalent to an initial value problem for the horospherical support function.
\begin{lemma}\label{horosupportevolution}
Suppose that $M$ is smooth strictly h-convex. The flow \eqref{smoothflow1} is equivalent  to the
following initial value problem.
\begin{align*}
\left\{\begin{aligned} 
\frac{\partial}{\partial t} \varphi & =-F(A^{-1}[\varphi])+\varphi\phi(t), \\ 
\varphi(\cdot,0) & =\varphi_\Omega(\cdot),\end{aligned}\right.
\end{align*}
    \end{lemma}

Let $D=\tilde{E}_{k+1}-\phi(t)\tilde{E}_k=(F-\phi(t))\tilde{E}_k$. By the evolution equations above, we have 
\begin{align}\label{Devolution}
    \frac{\partial}{\partial t}D=\dot{D}^{ij} \nabla_i\nabla_j F+(F-\phi(t))\dot{D}^{ij}((S^2)^{i}_j+2S^i_j)-\phi'(t)\tilde{E}_k.
\end{align}
Noticing that $\int_{M_t} Dd\mu_t=0$, we have 
\begin{align}\label{phievo}
    \frac{\partial}{\partial t}\phi(t)=\frac{1}{\int_{M_t} \tilde{E}_kd\mu_t}\left(\int_{M_t} (\phi(t)-F)HD+\dot{D}^{ij} \nabla_i\nabla_j F+(F-\phi(t))\dot{D}^{ij}((S^2)^{i}_j+2S^i_j)  d\mu_t\right)
\end{align}
In the following we need to estimate $\int_{M_t}|D| d\mu_t$. 
\begin{lemma}\label{intDevo} For a.e. $t\in[0,\infty)$, we have
    \begin{align*}
    &\frac{\partial}{\partial t} \int_{M_t}|D| d\mu_t\nonumber\\
    =&\int_{M_t}\mathrm{sgn}(D)\dot{D}^{ij} \nabla_i\nabla_j Fd\mu_t\nonumber\\
    &+\int_{M_t}\frac{\mathrm{sgn}(D)\int_{M_t}\tilde{E}_kd\mu_t-\int_{M_t}\mathrm{sgn}(D)\tilde{E}_kd\mu_t}{\int_{M_t}\tilde{E}_k d\mu_t}\\&((F-\phi(t))\dot{D}^{ij}((S^2)^{i}_j+2S^i_j)+(\phi(t)-F)HD)d\mu_t.
\end{align*}
\end{lemma}
\begin{proof}
    By the evolution equations above and the fact that $\tilde{E}$ is divergence free \cite{Reilly1973}, we have
\begin{align}
    &\frac{\partial}{\partial t} \int_{M_t}|D| d\mu_t\nonumber\\
    =&\int_{M_t} (\phi(t)-F)H|D| d\mu_t+\int_{M_t}\frac{\partial}{\partial t}|D| d\mu_t \nonumber\\
    =&\int_{M_t}(\phi(t)-F)H|D|+\mathrm{sgn}(D)(\dot{D}^{ij} \nabla_i\nabla_j F+(F-\phi(t))\dot{D}^{ij}((S^2)^{i}_j+2S^i_j)-\phi'(t)\tilde{E}_k) d\mu_t\nonumber\\
    =&\int_{M_t}(\phi(t)-F)H|D|\nonumber\\&+\frac{\mathrm{sgn}(D)\int_{M_t}\tilde{E}_kd\mu_t-\int_{M_t}\mathrm{sgn}(D)\tilde{E}_kd\mu_t}{\int_{M_t}\tilde{E}_k d\mu_t}(\dot{D}^{ij} \nabla_i\nabla_j F+(F-\phi(t))\dot{D}^{ij}((S^2)^{i}_j+2S^i_j))d\mu_t\nonumber\\&-\frac{\int_{M_t}\mathrm{sgn}(D)\tilde{E}_kd\mu_t}{\int_{M_t}\tilde{E}_k d\mu_t}\int_{M_t}(\phi(t)-F)HDd\mu_t\nonumber\\
    =&\int_{M_t}\frac{\mathrm{sgn}(D)\int_{M_t}\tilde{E}_kd\mu_t-\int_{M_t}\mathrm{sgn}(D)\tilde{E}_kd\mu_t}{\int_{M_t}\tilde{E}_k d\mu_t}\\&(\dot{D}^{ij} \nabla_i\nabla_j F+(F-\phi(t))\dot{D}^{ij}((S^2)^{i}_j+2S^i_j)+(\phi(t)-F)HD)d\mu_t\nonumber\\
    =&\int_{M_t}\mathrm{sgn}(D)\dot{D}^{ij} \nabla_i\nabla_j Fd\mu_t\nonumber\\
    &+\int_{M_t}\frac{\mathrm{sgn}(D)\int_{M_t}\tilde{E}_kd\mu_t-\int_{M_t}\mathrm{sgn}(D)\tilde{E}_kd\mu_t}{\int_{M_t}\tilde{E}_k d\mu_t}\\&((F-\phi(t))\dot{D}^{ij}((S^2)^{i}_j+2S^i_j)+(\phi(t)-F)HD)d\mu_t\nonumber
\end{align}
\end{proof}
\subsection{Convolutions on $\mathbb{S}^n$}

In this subsection, following \cite{Gri1999, Schuster2007}, we introduce the convolution on $\mathbb S^n$.

Consider the standard Lie group $SO(n+1)$ and its subgroup $SO(n)$. It is known that $\mathbb{S}=SO(n+1)/SO(n)$ and we can give an identification of $\mathbb{S}^n$ with $SO(n+1)/SO(n)$ for $u\in\mathbb{S}^n$ by $$u=g e_{\text{north}}\rightarrow gSO(n).$$ In the following we denote $G=SO(n+1)$ and $H=SO(n)$. Consider a sequence $\{f_m\}_{m\in \mathbb{N}}$ of smooth functions  on $SO(n+1)$ such that 

(1). $f_m\geq 0$ and $\int_{SO(n+1)} f_m(v) d v=1$.

(2). $\mathrm{spt}(f_m)\subset B(\frac{1}{m}),$ where $B(r)$ is the geodesic ball with radius $r$ centered at the unit element of $SO(n+1)$.

Denote by $C(\mathbb{S}^n)$ and $\mathcal{M}(\mathbb{S}^n)$ the collections of continuous functions and finite signed measures on $\mathbb{S}^n$ respectively. For any $\varphi\in C(\mathbb{S}^n)$ and  $\mu\in\mathcal{M}(\mathbb{S}^n)$, denote 
$$\langle\varphi, \mu\rangle=\int_{\mathbb{S}^n} \varphi(v) d \mu(v).$$

For any $\varphi\in C(\mathbb{S}^n)$, define the convolution of $\varphi$ as 
\begin{align*}
    f_m*\varphi(u)=\int_{SO(n+1)} \varphi(v^{-1}u)f_m(v) dv,\, \text{for} \,u\in \mathbb{S}^n.
\end{align*}

We know that $f_m*\varphi$ is smooth and converges to $\varphi$ in the uniform topology as $m\rightarrow \infty$.

For any $\mu\in\mathcal{M}(\mathbb{S}^n)$, we define the convolution of $\mu$ as 
\begin{align*}
    \langle f_m*\mu, \varphi\rangle=\int_{SO(n+1)}f_m(v)\int_{\mathbb{S}^n} \varphi(vu) d\mu(u) dv.
\end{align*}
When $\mu$ is absolutely continuous with respect to the Lebesgue measure on $\mathbb{S}^n$, there exists $\mu^*\in L^1(\mathbb{S}^n)$ such that  \begin{align*}
    \mu(u)=\mu^*(u) du.
\end{align*}
Then we have 
\begin{align*} \langle f_m*\mu, \varphi\rangle&=\int_{SO(n+1)}f_m(v)\int_{\mathbb{S}^n} \varphi(vu) d\mu(u) dv\\
&=\int_{SO(n+1)}f_m(v)\int_{\mathbb{S}^n} \varphi(vu)\mu^*(u) du dv\\
&=\int_{SO(n+1)}f_m(v)\int_{\mathbb{S}^n} \varphi(u)\mu^*(v^{-1}u) du dv\\
&=\int_{\mathbb{S}^n}\varphi(u)\int_{SO(n+1)}f_m(v) \mu^*(v^{-1}u) dv du\\
&=\int_{\mathbb{S}^n}\varphi(u)f_m*\mu^*(u) du.
\end{align*}
\section{Some properties of horospherical support functions and smooth approximation of h-convex bodies}\label{smoothsection}
We first introduce some properties of horospherical support functions.
\subsection{Some properties of horospherical support functions}
Let $\Omega\subset \mathbb{H}^{n+1}$ be a compact domain. We extend $\varphi_\Omega$ to $\mathbb{R}^{n+1}$ by 
$$\overline{\varphi_\Omega}(x)=\sup\limits_{\xi\in\Omega}\{\xi_{n+2}|x|-(\xi_1,\dots,\xi_{n+1})x\}, \, x\in\mathbb{R}^{n+1}.$$
\begin{lemma}\label{convexitylemma}
    $\overline{\varphi_\Omega}$ is a real valued convex function.
\end{lemma}
\begin{proof}
    For any $t\in [0,1]$, \begin{align*}
        \overline{\varphi_\Omega}(tx+(1-t)y)=&\sup\limits_{\xi\in\Omega}\{\xi_{n+2}|tx+(1-t)y|-(\xi_1,\dots,\xi_{n+1})(tx+(1-t)y)\}\\
        \leq&\sup\limits_{\xi\in\Omega}\{t(\xi_{n+2}|x|-(\xi_1,\dots,\xi_{n+1})x)+(1-t)(\xi_{n+2}|y|-(\xi_1,\dots,\xi_{n+1})y)\}\\
        \leq& t\sup\limits_{\xi\in\Omega}\{\xi_{n+2}|x|-(\xi_1,\dots,\xi_{n+1})x\}\\&+(1-t)\sup\limits_{\xi\in\Omega}\{\xi_{n+2}|y|-(\xi_1,\dots,\xi_{n+1})y\}\\
        \leq& t\overline{\varphi_\Omega}(x)+(1-t)\overline{\varphi_\Omega}(y).
    \end{align*}
\end{proof}
Denote $\overline f_\xi=|x|f_\xi(\frac{x}{|x|}), x\in \mathbb{R}^{n+1}\setminus \{0\}$. Actually $\overline{\varphi_\Omega}(x)=\sup\limits_{\xi\in\Omega} \overline f_{\xi}(x)$.

 We say that an h-convex domain $\Omega$ is  strictly h-convex in the sense of supporting horospheres if  every support horosphere intersects with its boundary at exactly one point. 
When $\Omega$ is strictly h-convex in the sense of supporting horospheres, it is easy to know that for any $e\in\mathbb{S}^n$ there exists a unique $x\in\partial \Omega$ such that  $f_x(e)=\varphi_\Omega(e).$
As in the Euclidean case, we derive the relation between strict h-convexity   and $C^1$ regularity of the horospherical support function.
\begin{lemma}
    For any $\Omega\in\mathcal{H}(\mathbb{H}^{n+1})$, $\Omega$ is strictly h-convex in the sense of supporting horospheres iff $\varphi_\Omega\in C^1$.
\end{lemma}
\begin{proof}
    By the convexity of $\overline{\varphi_\Omega}$, $\varphi_\Omega\in C^1$ iff $\overline{\varphi_\Omega}$ is differentiable in $\mathbb{R}^{n+1}\setminus \{0\}$.  
For any $x\in\mathbb{R}^{n+1}\setminus\{0\}$, consider the subdifferential of $\overline{\varphi_\Omega}$ at $x$, which is defined by
\begin{align*}
\partial \overline{\varphi_\Omega}(x)=\{v\in\mathbb{R}^{n+1}: \overline{\varphi_\Omega}(y)\geq \overline{\varphi_\Omega}(x)+v(y-x)\}.
\end{align*}
By classical convex analysis we have that $\overline{\varphi_\Omega}$ is differentiable at $x$ iff $\partial\overline{\varphi_\Omega}(x)$ is a set of single point.  Moreover, the strict h-convexity in the sense of supporting horospheres of $\Omega$ is equivalent to the condition that for any $x\in\mathbb{R}^{n+1}\setminus\{0\}$ there exists a unique $\xi^x\in\Omega$ such that 
$\overline{\varphi_\Omega}(x)=\overline f_{\xi^x}(x)$. Notice that $\overline{\varphi_\Omega}(x)=\sup\limits_{\xi\in\Omega}\overline f_\xi(x)$ and $f_\xi(x)$ is smooth with respect to $x$ and $\xi$. Hence, by Danskin's Theorem \cite{Danskin1967}, we have $$\partial\overline{\varphi_\Omega}(x)=\mathrm{conv}\{\overline{\nabla}_{\mathbb{R}^{n+1}}\overline f_{\xi}(x): \overline f_{\xi}(x)=\overline{\varphi_\Omega}(x)\}.$$
Noting that for any $\xi\in\mathbb{H}^{n+1}$ there holds
$$\overline{\nabla}_{\mathbb{R}^{n+1}} \overline f_{\xi}(x)=\frac{x}{|x|}\xi_{n+2}-(\xi_1,\dots, \xi_{n+1}),$$
we complete the proof.
\end{proof}
Now, we introduce that any compact set in hyperbolic space can correspond to  a Radon measure valued $(0,2)$-tensor on $\mathbb{S}^n$. 
By Lemma \ref{convexitylemma}, we have that $\varphi_\Omega$ is a real valued semi-convex function on $\mathbb{S}^n$. Hence we have $D\varphi_\Omega\in BV_{loc}(\mathbb{S}^n)$ and we can verify the Hessian of $\varphi_\Omega$ as a Radon measure valued $(0,2)$-tensor field $[\overline{\nabla}^2\varphi_\Omega]$ on $\mathbb{S}^n$ \cite{Evansbook}. By the Radon-Nikodym  theorem, we have 
  $$[\overline{\nabla}^2\varphi_\Omega]=[\overline{\nabla}^2\varphi_\Omega]_{ac}+[\overline{\nabla}^2\varphi_\Omega]_s,$$
where $[\overline{\nabla}^2\varphi_\Omega]_{ac}$, $[\overline{\nabla}^2\varphi_\Omega]_s$ are respectively the absolutely continuous part and singular part of $[\overline{\nabla}^2\varphi_\Omega]$. By the semi-convexity of $\varphi_\Omega$, we have that $[\overline{\nabla}^2\varphi_\Omega]_s$ is nonnegative.
 Moreover, due to the semi-convexity of $\varphi_\Omega$, by Alexandrov's theorem we have for $\mathcal{L}^n$-a.e. $u\in \mathbb{S}^n$ the Hessian $\overline{\nabla}^2\varphi_\Omega$ exists at $u$ in the Alexandrov sense. By the approximate differentiability of  functions of bounded variation   \cite{Evansbook}, we have $$[\overline{\nabla}^2\varphi_\Omega]_{ac}=\overline{\nabla}^2\varphi_\Omega d\mathcal{L}^n.$$

By the argument above, we associate to each compact domain $\Omega\subset\mathbb H^{n+1}$  a Radon measure valued $(0,2)$-tensor field as 
\begin{align*}
    A[\varphi_\Omega]:=[\overline{\nabla}^2 \varphi_\Omega]+\left(\frac{\varphi_\Omega-\varphi_\Omega^{-1}}{2}-\frac{|\overline{\nabla}\varphi_\Omega|^2}{2\varphi_\Omega}\right)\overline{g} d\mathcal{L}^n:=A[\varphi_\Omega]_{ac}+A[\varphi_\Omega]_{s},
\end{align*}
where \begin{align*}
    A[\varphi_\Omega]_{ac}&=\left\{\overline{\nabla}^2 \varphi_\Omega+\left(\frac{\varphi_\Omega-\varphi_\Omega^{-1}}{2}-\frac{|\overline{\nabla}\varphi_\Omega|^2}{2\varphi_\Omega}\right)\overline{g}\right\} d\mathcal{L}^n,\\
    A[\varphi_\Omega]_{s}&=[\overline{\nabla}^2\varphi_\Omega]_s\geq 0.
\end{align*}
By the calculus in \cite[Proposition 5.1]{AndrewsChenWei2021}, we obtain the concavity of $A$.
\begin{lemma}\label{concaveofA}
    For positive semi-convex functions $\varphi_1, \varphi_2$ and $t\in[0,1]$, we have
    \begin{align*}
        A[t\varphi_1+(1-t)\varphi_2]=&tA[\varphi_1]+(1-t)A[\varphi_2]\\&+t(1-t)\frac{|\varphi_1\overline{\nabla}\varphi_2-\varphi_2\overline{\nabla}\varphi_1|^2+|\varphi_1-\varphi_2|^2}{2\varphi_1\varphi_2(t\varphi_1+(1-t)\varphi_2)}g d\mathcal{L}^n\\
        \geq& tA[\varphi_1]+(1-t)A[\varphi_2].
    \end{align*}
\end{lemma}
Moreover, we have that $A$ is nonnegative on $\mathcal{C}(\mathbb{H}^{n+1})$.
\begin{lemma}
    For any compact domain $\Omega\subset\mathbb H^{n+1}$, we have $$A[\varphi_\Omega]\geq 0.$$
\end{lemma}
\begin{proof}
    By the argument above, we have $A[\varphi_\Omega]_{s}\geq 0$. Now we consider the absolutely continuous part $A[\varphi_\Omega]_{ac}$. For any $u\in \mathbb{S}^n$ such that $\overline{\nabla}^2 \varphi_\Omega$ exists in the Alexandrov sense, by the compactness of $\Omega$ there exists $p\in\Omega$ such that $$\varphi_\Omega(u)=\varphi_{\{p\}}(u).$$
    By the definition of $\varphi_\Omega$ we have 
    \begin{align*}
        \varphi_\Omega(v)\geq\varphi_{\{p\}}(v), v\in\mathbb{S}^n.
    \end{align*}
       Hence we have 
    \begin{align*}
        \overline{\nabla}\varphi_\Omega(u)=&\overline{\nabla}\varphi_{\{p\}}(u),\\
        \overline{\nabla}^2\varphi_\Omega(u)\geq&\overline{\nabla}^2\varphi_{\{p\}}(u).
    \end{align*}
     By the calculations in \cite{LiXu2022}, there holds $$A[\varphi_{\{p\}}(v)]=0, v\in\mathbb{S}^n.$$
    Hence we have $A(\varphi_\Omega)(u)\geq A(\varphi_{\{p\}})(u)=0$, which implies 
    $$A[\varphi_\Omega]_{ac}\geq 0.$$
\end{proof}
Actually, for any $\Omega\in\mathcal{C}(\mathbb{H}^{n+1})$, $u_{\varphi_\Omega}:=\log \varphi_\Omega$ is a horospherical support function of the h-convex hull of $\Omega$.

\subsection{Convolution of $A$}
Now we consider the convolution properties of $A$. Denote by $\varphi_\Omega^m$ the convolution $f_m*\varphi_\Omega$. We claim the following property holds.
\begin{thm}\label{positiveofconvolution}
    For any $\Omega\in\mathcal{C}(\mathbb{H}^{n+1})$, we have 
    \begin{align*}
        A[\varphi_\Omega^m]\geq f_m*A[\varphi_\Omega]\geq 0.
    \end{align*}
\end{thm}
\begin{proof}
Let $\phi$ and $X$ be a smooth positive function and a smooth vector field on $\mathbb{S}^n$, respectively.  We prove that $$\langle A[\varphi_\Omega^m], X\otimes X\rangle\geq \langle f_m*A[\varphi_\Omega], X\otimes X\rangle\geq 0.$$

Denote \begin{align*}
    A_1[\varphi_\Omega]:=&[\overline{\nabla}^2 \varphi_\Omega],\\
A_2[\varphi_\Omega]:=&\left(\frac{\varphi_\Omega-\varphi_\Omega^{-1}}{2}-\frac{|\overline{\nabla}\varphi_\Omega|^2}{2\varphi_\Omega}\right)\overline{g} d\mathcal{L}^n.
\end{align*}

We first prove $\langle A_1[\varphi_\Omega^m], X\otimes X\rangle= \langle f_m*A_1[\varphi_\Omega], X\otimes X\rangle$.
\begin{align*}
    \langle f_m*A_1[\varphi_\Omega], X\otimes X\rangle=&\int_{SO(n+1)}f_m(v)\int_{\mathbb{S}^n} \overline{\nabla}^2\varphi_\Omega(vX\otimes vX) dv\\
    =&-\int_{SO(n+1)}f_m(v)\int_{\mathbb{S}^n}\overline{\nabla}_i(vX^ivX^j)(u) \overline{\nabla}_j\varphi_\Omega(u)dudv\\    =&\int_{SO(n+1)}f_m(v)\int_{\mathbb{S}^n}\overline{\nabla}_j\overline{\nabla}_i(vX^ivX^j)(u) \varphi_\Omega(u) du dv\\    =&\int_{SO(n+1)}f_m(v)\int_{\mathbb{S}^n}\overline{\nabla}_j\overline{\nabla}_i(X^iX^j)(u) \varphi_\Omega(v^{-1}u) du dv\\    
    =&\int_{\mathbb{S}^n}\overline{\nabla}_j\overline{\nabla}_i(X^iX^j)(u)\int_{SO(n+1)}f_m(v) \varphi_\Omega(v^{-1}u) dv du\\    
    =&\int_{\mathbb{S}^n}\overline{\nabla}_j\overline{\nabla}_i(X^iX^j)(u)f_m*\varphi_\Omega(u) du\\
    =&\int_{\mathbb{S}^n}(X^iX^j)(u)\overline{\nabla}_i\overline{\nabla}_j f_m*\varphi_\Omega(u) du\\
    =&\langle A_1[\varphi^m_\Omega], X\otimes X\rangle,
\end{align*}
where the second and third equalities hold by the regularity properties of $\varphi_\Omega$ that we mentioned above and the seventh equality holds by the smoothness of $f_m*\varphi_\Omega$.

Now we prove that $\langle A_2[\varphi_\Omega^m], X\otimes X\rangle\geq \langle f_m*A_2[\varphi_\Omega], X\otimes X\rangle$. 
By definition, we have 
\begin{align*}
    \langle A_2[\varphi_\Omega^m], X\otimes X\rangle=\int_{\mathbb{S}^n}\left(\frac{\varphi^m_\Omega-(\varphi_\Omega^m)^{-1}}{2}-\frac{|\overline{\nabla}\varphi^m_\Omega|^2}{2\varphi^m_\Omega}\right)(u)\overline{g}(X,X)(u)du
\end{align*}
and
\begin{align*}
    &\langle f_m*A_2[\varphi_\Omega], X\otimes X\rangle\\=&\int_{SO(n+1)}f_m(v)\int_{\mathbb{S}^n}\left(\frac{\varphi_\Omega-\varphi_\Omega^{-1}}{2}-\frac{|\overline{\nabla}\varphi_\Omega|^2}{2\varphi_\Omega}\right)(u)\overline{g}(vX,vX)(u)du dv\\
    =&\int_{SO(n+1)}f_m(v)\int_{\mathbb{S}^n}\left(\frac{\varphi_\Omega-\varphi_\Omega^{-1}}{2}-\frac{|\overline{\nabla}\varphi_\Omega|^2}{2\varphi_\Omega}\right)(v^{-1}u)\overline{g}(X,X)(u)du dv\\
    =&\int_{\mathbb{S}^n}\int_{SO(n+1)}f_m(v)\left(\frac{\varphi_\Omega-\varphi_\Omega^{-1}}{2}-\frac{|\overline{\nabla}\varphi_\Omega|^2}{2\varphi_\Omega}\right)(v^{-1}u)dv \overline{g}(X,X)(u) du
\end{align*}
By the calculations in Lemma \ref{concaveofA}, we have $A_2(t\varphi_1+(1-t)\varphi_2)\geq tA_2(\varphi_1)+(1-t)A_2(\varphi_2)$ for positive locally Lipschitz functions $\varphi_1$, $\varphi_2$ and $t\in[0,1]$.  Hence, by Jensen's inequality, we have $$ A_2(\varphi^m_{\Omega})\geq f_m *A_2(\varphi_\Omega),$$ i.e. 
\begin{align*}
    \left(\frac{\varphi^m_\Omega-(\varphi_\Omega^m)^{-1}}{2}-\frac{|\overline{\nabla}\varphi^m_\Omega|^2}{2\varphi^m_\Omega}\right)(u)\geq\int_{SO(n+1)}f_m(v)\left(\frac{\varphi_\Omega-\varphi_\Omega^{-1}}{2}-\frac{|\overline{\nabla}\varphi_\Omega|^2}{2\varphi_\Omega}\right)(v^{-1}u)dv.
\end{align*}
Hence we have $\langle A_2[\varphi_\Omega^m], X\otimes X\rangle\geq \langle f_m*A_2[\varphi_\Omega], X\otimes X\rangle$.

Noticing that 
\begin{align*}
    \langle f_m*A[\varphi_\Omega], X\otimes X\rangle=\int_{SO(n+1)}f_m(v)\int_{\mathbb{S}^n}vX\otimes vX d A[\varphi_\Omega]\geq 0,
\end{align*}we complete the proof.
\end{proof}
\subsection{Smooth approximation of h-convex bodies and more on horospherical support function}
By the arguments above, as in the argument in \cite{Schneider} we obtain the existence of smooth h-convex approximation for h-convex bodies.
\begin{proof}[Proof of Lemma \ref{smoothapproxiamate}] 
We first prove the existence of outer approximation.
Let $\Omega\in\mathcal{H}(\mathbb{H}^{n+1})$. For any $\varepsilon>0$, by Lemma \ref{outerhconvex} we have $\Omega^\varepsilon\in\mathcal{H}(\mathbb{H}^{n+1})$. 
Define $\varphi^{m,\varepsilon}_\Omega=f_m*\varphi_{\Omega^\epsilon}$. By Theorem \ref{positiveofconvolution}, we have that for any $\varepsilon>0$ and $m\in \mathbb{N}$, $\varphi^{m,\varepsilon}_\Omega$ is smooth and $A(\varphi^{m,\varepsilon}_\Omega)\geq 0$. For any $\eta>0$, inspired by the proof of \cite[Lemma 3.1]{LuoWei2025}, we have $A(\varphi^{m,\varepsilon}_\Omega+\eta)>0$. Hence, by the argument in \cite{AndrewsChenWei2021}, there exists a smooth h-convex domain $\Omega^{m,\varepsilon}_\eta$ such that $\varphi_{\Omega^{m,\varepsilon}_\eta}=\varphi^{m,\varepsilon}_\Omega+\eta$. By the properties of convolution we have $\varphi_{\Omega^{m,\varepsilon}_\eta}\rightarrow\varphi_{\Omega^\varepsilon}$ in the uniform topology as $m\rightarrow\infty$, $\eta\rightarrow 0^+$. Hence, by \cite[Corollary 3.2]{LiWanXv2024}, we have $\Omega^{m,\varepsilon}_\eta\rightarrow\Omega^\varepsilon$ in the  Hausdorff sense as $m\rightarrow\infty$, $\eta\rightarrow 0^+$. Then,  letting $\varepsilon\rightarrow 0^+$ and considering the smooth strictly h-convex approximation for $\Omega^\varepsilon$, we complete the proof.

For inner approximation, we first choose an interior point of $\Omega$ as the origin of the hyperbolic space in Minkowski space. 
Notice that the intersection of the $\varepsilon$ inner parallel sets of a class of horoballs is equal to the  $\varepsilon$ inner parallel set of the intersection of this class of horoballs. Hence, for $\varepsilon$ small enough, $\Omega_\varepsilon$ is h-convex. 
Hence, we can consider the smooth approximation of $\Omega_\varepsilon$. 
\end{proof}

For any $\Omega\in\mathcal{H}(\mathbb{H}^{n+1})$, by the semi-convexity of $\varphi_\Omega$ and the approximation result, we have that  the modified area measure can be expressed as
\begin{align}\label{areameasure}
    &\tilde{S}_0(\Omega,\beta)=\int_{\mathbb{S}^n\cap\beta}\varphi_\Omega^{-n}d\mathcal{L}^n\nonumber\\
    &\tilde{S}_1(\Omega,\beta)=\int_{\mathbb{S}^n\cap\beta}\varphi_\Omega^{-n+1}\mathrm{tr}(A[\varphi_\Omega]).
\end{align}
The smooth approximation Lemma \ref{smoothapproxiamate} immediately implies that the inequality parts in Theorem \ref{A-F} hold. Moreover, using Lemma \ref{smoothapproxiamate}, we give a characterization of horospherical support functions.
\begin{lemma}    
    For any semi-convex function $\varphi>0$ on $\mathbb{S}^{n}$, $\log\varphi$ is the horospherical support function of $\Omega_\varphi:=\{X\in\mathbb{H}^{n+1}: -\langle X, (\xi,1)\rangle_L\leq \varphi(\xi), \xi\in\mathbb{S}^n\}$ iff $A[\varphi]\geq 0$.
\end{lemma}
\begin{proof}
By the arguments above we only need to prove that $A[\varphi]\geq 0$ implies  $\Omega_\varphi:=\{X\in\mathbb{H}^{n+1}: -\langle X, (\xi,1)\rangle_L\leq \varphi(\xi), \xi\in\mathbb{S}^n\}$ is an  h-convex body or a set of a single point.  

If $A[\varphi]\geq 0$, by the arguments above there exists a sequence of smooth functions $\varphi_j>\varphi_{j+1}>\varphi$ on $\mathbb{S}^{n}$ such that $A[\varphi_j]>0$ and $\varphi_j\rightarrow\varphi$ as $j\rightarrow \infty$. By \cite{AndrewsChenWei2021}, we have that $\Omega_{\varphi_j}$ is h-convex for any $j\in\mathbb{N}$ and $\Omega_{\varphi_j}\rightarrow\Omega_\varphi$ in the Hausdorff sense as $j\rightarrow\infty$.
 If $\Omega_{\varphi_j}$ have uniform lower bounds on the inner radius, by Theorem \ref{hausdorffofhconvex}, we have that $\Omega_\varphi$ is h-convex. 
 If the inner radius of $\Omega_{\varphi_j}$ tends to $0$ as $j\rightarrow\infty$, noticing that $\Omega_{\varphi_{j+1}}\subseteq\Omega_{\varphi_j}$, there exists $p\in \Omega_{\varphi_j}$. Moreover, we have $\varphi_{\{p\}}\leq \varphi_j$ for any $j\in\mathbb{N}$ and $\varphi_{\{p\}}+\eta>\varphi$ for any $\eta>0$, from which we have $$\Omega_\varphi:=\{X\in\mathbb{H}^{n+1}: -\langle X, (\xi,1)\rangle_L\leq \varphi(\xi), \xi\in\mathbb{S}^n\}=\{p\}.$$
\end{proof}

We next consider the smooth strictly h-convex approximation of a h-convex hypersurface $M=\partial\Omega$ of class $C^{1,1}$ in the sense of $C^{1,\alpha}\cap W^{2,p}$ for any $0\leq\alpha<1$ and $p\in \mathbb{N}$. This strong approximation will play a role in the proof of Theorem \ref{metricball}. If the h-convex body $\Omega$ satisfies \eqref{conditionsofrigisupp}, combining the strong approximation with the outer parallel construction, we obtain a sequence of  smooth strictly h-convex approximations such that the total variation norms of the signed measure associated with the proportionality relation tend to $0$.

\begin{lemma}\label{regularityofpar}
$M^t=\partial \Omega^t\rightarrow M$ in the sense of $C^{1,\alpha}\cap W^{2,p}$ as $t\rightarrow 0^+$.\end{lemma}
\begin{proof}
Fix $o\in\mathring{\Omega}$ as the center. We consider the warped product construction of $\mathbb{H}^{n+1}$.

Suppose that $M=graph(\theta,\rho(\theta))$ and $M^t=(\theta,\rho_t(\theta))$.
   Firstly, by the Hausdorff convergence of $M_t$, we immediately obtain that $\rho_t\rightarrow \rho$ in the sense of $C^0$.

   By the h-convexity of $M$ and $M_t$, we have that $\langle \nu,\partial r\rangle$ has a uniform lower bound in $[0,t)$, from which we obtain the uniform upper bound on $|\overline{\nabla}\rho_t|$.
   
   By the calculation in \cite{Kohlmann1991}, we have that  the second fundamental form of $M^{t}$ is uniformly bounded in the sense of $\mathcal{L}^\infty$ on $[0,t_0)$, from which we have that $|\overline{\nabla^2}\rho_t|$ has a uniform bound in the  sense of $\mathcal{L}^\infty$  on $[0,t_0)$. Hence, by $\rho_t\in C^{1,1}$, we have that  $\rho_t$   has a uniform bound in the $C^{1,1}$ sense on $[0,T)$.

   By interpolation inequalities \cite[Lemma 6.35]{GT}, we have 
   $$|\rho_t-\rho|_{C^{1,\alpha}}\leq C(\varepsilon)|\rho_t-\rho|_{C^0}+\varepsilon |\rho_t-\rho|_{C^{1,1}}.$$

Now we try to prove that $\overline{\nabla}^2\rho_t\rightarrow \overline{\nabla}^2\rho$ in the  sense of $L^{p}$ as $t\rightarrow 0^+$. Actually, noticing that $\Omega^t$ is an outer parallel set of $\Omega$, we can calculate precisely.  Let $R(\theta)=\exp_0{(\rho(\theta)\theta)}$ and $R_t(\theta)=\exp_o(\rho_t(\theta)\theta)$.
Denote by $A, A_t$  the second fundamental forms  of $M, M_t$, respectively. Let $h=R^*A$ and $h_t=R_t^*A_t$.  By a direct calculation, we have 
\begin{align}
    \overline{\nabla}^2\rho_t-\overline{\nabla}^2\rho=o(1)-\sqrt{1+\frac{|\overline{\nabla}\rho_t|^2}{\sinh^2\rho_t}}h_t+\sqrt{1+\frac{|\overline{\nabla}\rho|^2}{\sinh^2\rho}}h.
\end{align}
Hence we need to estimate $|h_t-h|$.

Let $\varphi_t(x)=\exp_x{(t\nu(x))}$. By the calculation in \cite{Kohlmann1991}, we have 
\begin{align}\label{dphi}d\varphi_t(u)=P_{t\nu(x)}(\cosh t u+\sinh t DN(u)),\end{align}
where $P$ is parallel transport.  

The eigenvalue $k_i(t)$ of $A_t$ at $\varphi_t(x)$ is related to the eigenvalue $k_i$ of $A$ at $x$ by  
$$k_i(t)=\frac{\sinh t+k_i\cosh t}{\cosh t+k_i\sinh t}.$$

Choose the orthonormal frame $\{u_i\}$ of $T_xM$ such that $A(u_i,u_i)=k_i$. We have 
\begin{align*}\varphi_t^*A_t(u_i,u_j)&=(\cosh t+\sinh t k_i)(\sinh t+\cosh t k_i)\delta_{ij}\\
&=(\sinh t\cosh t+(\sinh^2t+\cosh^2t)k_i+\sinh t\cosh t k_i^2)\delta_{ij}
\end{align*}
Hence we have 
\begin{align*}
    \varphi_t^*A_t(u_i,u_j)-A(u_i,u_j)=(\cosh^2t-1)k_i\delta_{ij}+(\sinh t\cosh t+\sinh^2tk_i+\sinh t\cosh t k_i^2)\delta_{ij}.
\end{align*}
Hence we have $\varphi_t^* A_t\rightarrow A$ in the sense of $\mathcal{L}^\infty$.

Let $F_t(\theta)=\varphi_t(\exp_o(\rho(\theta)\theta))$. Then  $F_t$ is  bi-Lipschitz from $\mathbb{S}^n$ to $M_t$. Denote by  $\pi:\mathbb{H}^{n+1}\setminus\{o\}\rightarrow \mathbb{S}^{n}$ the direction projection, which is smooth. Define $\theta_t:\mathbb{S}^n\rightarrow \mathbb{S}^n$ as $\theta_t(\theta)=\pi(F_t(\theta))$. Hence we have that $\theta_t$ is Lipschitz and bijective.  Moreover, we have $\rho_t(\theta)=d(o,F_t(\theta_t^{-1}(\theta)))=d(o,R_t(\theta))$. 

Noticing that $R_t=F_t\circ\theta_t^{-1}$, hence we have that $h_t={\theta_t^{-1}}^*F_t^*A_t={\theta_t^{-1}}^*R^*\varphi_t^*A_t$. For any smooth vector field $X$ on $\mathbb{S}^n$, 
we have \begin{align*}
&(h_t-h)(X,X)(u)\\=&{\theta_t^{-1}}^*R^*\varphi_t^*A_t(X,X)(u)-R^*\varphi_t^*A_t(X,X)(u)+R^*\varphi_t^*A_t(X,X)(u)-h(X,X)(u)\\
=&{\theta_t^{-1}}^*R^*\varphi_t^*A_t(X,X)(u)-R^*\varphi_t^*A_t(X,X)(u)+R^*\varphi_t^*A_t(X,X)(u)-R^*A(X,X)(u)\\
=&R^*\varphi_t^*A_t|_{\theta_t^{-1}(u)}(d{\theta_t^{-1}}X(u),d{\theta_t^{-1}}X(u))-R^*\varphi_t^*A_t(X,X)(u)+R^*\varphi_t^*A_t(X,X)(u)\\&-R^*A(X,X)(u).
\end{align*}
To obtain the $L^p$ estimates of $h_t-h$, we only need to prove that $\theta_t^{-1}\rightarrow Id_{\mathbb{S}^n}$ and $d\theta_t^{-1}\rightarrow Id_{T\mathbb{S}^n}$ in  the  sense  of $\mathcal{L}^\infty$ under the Sasaki metric.

By the hyperbolic Cosine theorem, we can obtain that $\theta_t\rightarrow Id_{\mathbb{S}^n}$ as $t\rightarrow 0^+$ in the $C^0$ sense immediately. Hence we have $\theta_t^{-1}\rightarrow Id_{\mathbb{S}^n}$ as $t\rightarrow 0^+$ in the $C^0$ sense.

By \eqref{dphi} and the $C^{1,1}$ regularity of $R$, we have 
\begin{align*}
    &d\theta_t(X)(u)=d\pi\circ d\varphi_t(dR(X))=d\pi( P_{t\nu(\Pi(dR(X)))}(\cosh t dR(X)+\sinh tDN(dR(X)))\\
    \rightarrow &d\pi(dR(X))=X(u)
\end{align*} in the sense of $\mathcal{L}^\infty$ under the Sasaki metric.

By the uniform bound of the principal curvature, we have $d\varphi_t(x)\rightarrow Id_{T_xM}$ in the sense of $\mathcal{L}^\infty$ under the Sasaki metric. 

 Hence, we have 
$d\theta_t=d\pi\circ d\varphi_t\circ dR\rightarrow d\pi\circ d R=Id_{T\mathbb{S}^n}$ in the sense of $\mathcal{L}^\infty$. Hence we have $d\theta_t^{-1}\rightarrow Id_{T\mathbb{S}^n}$ in the sense of $\mathcal{L}^\infty$ under the Sasaki metric. In the following we assume that for any  $(0,2)$-tensor $T\in \mathcal{L}^\infty(\mathbb{S}^n)$ there holds $|{\theta^{-1}_t}^*T|_{\mathcal{L}^\infty}\leq 2|T|_{\mathcal{L}^\infty}$  without loss of generality.

Noticing  that  $C^0(\mathbb{S}^n)$  is dense in $L^p(\mathbb{S}^n)$, let $\{R_{m}\}$ be a sequence of continuous $(0,2)$-tensors on $\mathbb{S}^n$ such that  $R_{m}\rightarrow R^*A$ in the sense of $L^p$ as $m\rightarrow\infty$.

 By the previous argument we have $d\varphi_t(x)\rightarrow Id_{T_xM}$ in the sense of $\mathcal{L}^\infty$. Hence, for any smooth vector field $X$ on $\mathbb{S}^n$ and $t>0$, we have 
\begin{align*}&||{\theta_t^{-1}}^*R^*\varphi_t^*A_t(X,X)-R^*\varphi_t^*A_t(X,X)||_{L^p}\\
\leq& (\int_{\mathbb{S}^n}|{\theta_t^{-1}}^*(R^*\varphi_t^*A_t-R^*A)(X,X))|^p du)^\frac{1}{p}\\
&+(\int_{\mathbb{S}^n}|{\theta_t^{-1}}^*(R^*A(X,X)-R_{m}(X,X))|^p du)^\frac{1}{p}\\
&+(\int_{\mathbb{S}^n}|{\theta_t^{-1}}^*R_{m}(X,X)-R_{m}(X,X)|^p du)^\frac{1}{p}\\
&+(\int_{\mathbb{S}^n}|(R_{m}(X,X)-R^*\varphi_t^*A_t(X,X))|^p du)^\frac{1}{p}.
\end{align*}
Hence we have \begin{align*}
    &\limsup_{t\rightarrow 0}||{\theta_t^{-1}}^*R^*\varphi_t^*A_t(X,X)-R^*\varphi_t^*A_t(X,X)||_{L^p}\\
    \leq& 2(\int_{\mathbb{S}^n}|(R^*A(X,X)-R_{m}(X,X))|^p du)^\frac{1}{p}.
\end{align*}
Letting $m\rightarrow \infty$, we complete our proof.
\end{proof}

\begin{lemma}\label{higherregularapproximation}
    For any $\Omega\in\mathcal{H}(\mathbb{H}^{n+1})$ with $M=\partial\Omega$ of class $C^{1,1}$, $0<\alpha< 1$ and $p\in\mathbb{N}$, there exists a sequence of smooth closed strictly h-convex hypersurfaces $\{M_i\}_{i\in\mathbb{N}}$ such that $M_i\rightarrow M$ in the sense of $C^{1,\alpha}\cap W^{2,p}$.
\end{lemma}
\begin{proof}
Step 1: Consider the Poincare ball model. By \cite[Theorem 1.3]{Scholtes13}, we have that $M$ is of positive reach. Hence for $\varepsilon$ small enough the Euclidean inner parallel hypersurface $M_\varepsilon$ is $C^{1,1}$. And by a direct calculation we have that $M_\varepsilon$ is uniformly strictly h-convex. By an argument similar to that in Lemma \ref{regularityofpar}, we have that   $M_\varepsilon\rightarrow M$ in the sense of $C^{1,\alpha}\cap W^{2,p}$.

More precisely, we 
consider the following approximation. Let $x\in M$ such that $\nu^H(x)$ is differentiable at $x$, where $\nu^H$ is the unit outer normal vector of $M$ at $x$. Let $\nu=\frac{2}{1-|x|^2}\nu^H$.
 Noticing the metrics are conformal, by the regularity of $M$, we have
$$k_i^\mathbb{H}(x)= \frac{1-|x|^2}{2}k_i^\mathbb{R}+\langle x,\nu\rangle.$$
Let $k_{i,\varepsilon}^\mathbb{R}$ and $k_{i,\varepsilon}^\mathbb{H}(x)$ be the Euclidean and hyperbolic principal curvatures of $M_\varepsilon$ respectively.
By the Riccati equation, we have 
$$k_{i,\varepsilon}^\mathbb{R}(x)=\frac{k_i^\mathbb{R}(x)}{1-\varepsilon k_i^\mathbb{R}(x)}.$$
Hence we have 
$$ k_{i,\varepsilon}^\mathbb{H}(x)-k_i^\mathbb{H}(x)=\frac{\varepsilon(1-|x|^2){k_i^{\mathbb{R}}}^2(x)+(2\varepsilon\langle x,\nu\rangle+\varepsilon^2)k_i^\mathbb{R}(x)-2\varepsilon}{2(1-\varepsilon k_i^\mathbb{R}(x))}.$$
By the h-convexity of $M$, we have 
$$(1-|x|^2)k_i^\mathbb{R}(x)\geq 2(1-\langle x,\nu\rangle)>C(M)>0.$$
Hence we have 
$$ k_{i,\varepsilon}^\mathbb{H}(x)-k_i^\mathbb{H}(x)\geq \frac{\varepsilon^2k^\mathbb{R}_i(x)}{2(1-\varepsilon k_i^\mathbb{R}(x))}.$$
By choosing $\varepsilon$ small enough, we complete step 1.

Step 2: By the argument above we assume that  $M$ is uniformly strictly h-convex. Consider the radial graph  representation $(\theta, \rho(\theta))$ of $M$. Consider the spherical convolution $\rho_m$. $M_m=graph(\theta,\rho_m(\theta))$ is smooth. We have $\rho_m\rightarrow\rho$ in the  sense of $C^{1,\alpha} \cap W^{2,p}$. Hence $M_m$ converges to $M$  in the sense of $C^{1,\alpha} \cap W^{2,p}$.

Step 3: For $m$ large enough, we have that $M_m$ is  strictly h-convex. 
More precisely, we can consider $\overline{\nabla}^2\rho$ as a Radon measure on $\mathbb{S}^n$ and absolutely continuous with respect to $d \mathcal{L}^n$, since $\rho\in C^{1,1}$.  

By the $C^{1,1}$ regularity of $M$, we have
$$  h_{i j}=\frac{1}{\sqrt{1+\frac{|\overline{\nabla}\rho|^2}{\sinh ^2\rho}}}\left(-\rho_{i j}+\sinh \rho \cosh \rho \sigma_{i j}+2\mathrm{coth} \rho \rho_i \rho_j\right)  $$
and 
$$g_{i j}=\rho_i \rho_j+\sinh ^2\rho \sigma_{i j}.$$

Hence we have 
\begin{align}\label{equihconvex}
    &h>(1+\delta) g\iff \nonumber\\
    &-\rho_{ij}>-\sinh \rho \cosh \rho \sigma_{i j}-2\mathrm{coth} \rho \rho_i \rho_j+(1+\delta)\sqrt{1+\frac{|\overline{\nabla}\rho|^2}{\sinh ^2\rho}}\left(\rho_i \rho_j+\sinh ^2\rho \sigma_{i j}\right).
\end{align}
By the $C^{1,1}$ regularity of $\rho$, we have $\overline{\nabla}(f_m*\rho)=f_m*\overline{\nabla}\rho$ and $\overline{\nabla}^2 (f_m*\rho)=f_m*\overline{\nabla}^2(\rho)$.

Hence, for any $\delta<C(M)$, let $C'(M)=\frac{C(M)+\delta}{2}$. We have 
\begin{align*}
    &\langle-\overline{\nabla}^2\rho^m,X\otimes X\rangle=\langle-f_m*\overline{\nabla}^2\rho,X\otimes X\rangle=-\int_{SO(n+1)} f_m(v)
    \int_{\mathbb{S}^n}\overline{\nabla}^2\rho(vX\otimes vX)d v\\
    >&\int_{SO(n+1)} f_m(v)\int_{\mathbb{S}^n}vX^ivX^j(u)\\
    &\left(-\sinh \rho \cosh \rho \sigma_{i j}-2\mathrm{coth} \rho \rho_i \rho_j+(1+C(M))\sqrt{1+\frac{|\overline{\nabla}\rho|^2}{\sinh ^2\rho}}\left(\rho_i \rho_j+\sinh ^2\rho \sigma_{i j}\right)\right)(u)dudv\\
   \rightarrow& \int_{\mathbb{S}^n}X^iX^j(u)\\
   &\left(-\sinh \rho \cosh \rho \sigma_{i j}
   -2\mathrm{coth} \rho \rho_i \rho_j+(1+C(M))\sqrt{1+\frac{|\overline{\nabla}\rho|^2}{\sinh ^2\rho}}\left(\rho_i \rho_j+\sinh ^2\rho \sigma_{i j}\right)\right)(u)du,\\& m\rightarrow\infty.
\end{align*}

Moreover, we have
 $\rho^m\rightarrow \rho$ and $\nabla \rho^m\rightarrow \nabla\rho$ uniformly as $m\rightarrow \infty$.
 Hence, by the smoothness of $\rho^m$, we have  
 \begin{align*}
     &-\overline{\nabla}^2\rho^m(X\otimes X)(u)\\
     \geq &\left(-\sinh \rho^m \cosh \rho^m \sigma_{i j}-2\mathrm{coth} \rho^m \rho^m_i \rho^m_j+(1+C'(M))\sqrt{1+\frac{|\overline{\nabla}\rho^m|^2}{\sinh ^2\rho^m}}\left(\rho^m_i \rho^m_j+\sinh ^2\rho^m \sigma_{i j}\right)\right)\\&X^iX^j(u)\\
     >&\left(-\sinh \rho^m \cosh \rho^m \sigma_{i j}-2\mathrm{coth} \rho^m \rho^m_i \rho^m_j+(1+\delta)\sqrt{1+\frac{|\overline{\nabla}\rho^m|^2}{\sinh ^2\rho^m}}\left(\rho^m_i \rho^m_j+\sinh ^2\rho^m \sigma_{i j}\right)\right)\\&X^iX^j(u)
 \end{align*} for $m$ large enough, 
 which completes the proof.
\end{proof}
We remark that Lemma \ref{higherregularapproximation} also yields  Lemma \ref{smoothapproxiamate} by considering the suitable parallel sets. We retain the proof in Section \ref{smoothsection}, as this approach also gives more details on the horospherical support functions.

\section{Stability for the Alexandrov-Fenchel inequality}\label{stabilityforaf}
In this section, we derive the Fraenkel asymmetry stability independent of curvature bounds described by the Alexandrov-Fenchel deficits for h-convex bodies.

\begin{proof}[Proof of Theorem \ref{stability}]
By Lemma \ref{smoothapproxiamate}, the inequality in Theorem \ref{A-F} holds.
Denote $\tilde{D}_{k,l}(\Omega)=\tilde{W}_k(\Omega)-\tilde{f}_k(\tilde{f}^{-1}_l(\tilde{W_l}(\Omega))$.
We only need to prove $CD_{k,l}\geq  \tilde{D}_{1,0}$. Let $r_i=f_i^{-1}(W_i(\Omega))$ and $\tilde{r}_i=\tilde{f}^{-1}_i(\tilde{W}_i(\Omega))$. If $l=0$, the proposition is trivial. In the following we assume that $1\leq l$.  By the inequality in Theorem \ref{A-F}, we have that $r_i$ and $\tilde{r}_i$ are increasing with respect to $i$. First, we have 
\begin{align*}
D_{k,l}(\Omega)&=\sum_{i=0}^k \binom{k}{i}\tilde{W}_i(\Omega)-\sum_{i=0}^k \binom{k}{i}\tilde{f}_i(r_0)+\sum_{i=0}^k \binom{k}{i}(\tilde{f}_i(r_0)-\tilde{f}_i(r_l))\\
&=\sum_{i=0}^k \binom{k}{i}\tilde{D}_{i,0}(\Omega)+\sum_{i=0}^k \binom{k}{i}(\tilde{f}_i(r_0)-\tilde{f}_i(r_l)).   
\end{align*}
Hence we have 
\begin{align*}
\tilde{D}_{1,0}(\Omega)\leq \sum_{i=0}^k \binom{k}{i}\tilde{D}_{i,0}(\Omega)= D_{k,l}(\Omega)+\sum_{i=0}^k \binom{k}{i}(\tilde{f}_i(r_l)-\tilde{f}_i(r_0)).   
\end{align*}
By a direct calculation (as in \cite{Zhang2026,LiXu2022}), we have 
\begin{align}\label{expressf}
   &f_i(r)=\omega_n\int_0^r \sinh^{n-i}t\cosh^i(t) dt\nonumber\\
   &\tilde{f}_i(r)=\omega_n\int_0^r \sinh^{n-i} t \,e^{-it} dt.
\end{align} 
Noticing that, by \eqref{expressf}, $\tilde{f}_i(r_l)-\tilde{f}_i(r_0)\leq C(i,\rho)(\tilde{f}_0(r_l)-\tilde{f}_0(r_0))$, hence we only need to estimate $\tilde{f}_0(r_l)-\tilde{f}_0(r_0)$ by $D_{k,l}(\Omega)$.

Consider \begin{align*}
    D_{k,l}(\Omega)=&\sum\limits_{i=0}^k \binom{k}{i}\tilde{W}_i(\Omega)-f_k(r_l)\\
    =&\sum\limits_{i=0}^k \binom{k}{i}\left(\tilde{W}_i(\Omega)-\tilde{f}_i(r_l)\right).
\end{align*}
Moreover, we have 
\begin{align*}
    0=W_l(\Omega)-f_l(r_l)=\sum\limits_{i=0}^l\binom{l}{i}\left(\tilde{W}_i(\Omega)-\tilde{f}_i(r_l)\right)=\sum\limits_{i=0}^l\binom{l}{i}\left(\tilde{f}_i(\tilde{r}_i)-\tilde{f}_i(r_l)\right),
\end{align*}
from which we have that there exists  $1\leq a\leq l$ such that $\tilde r_a\geq r_l\geq \tilde{r}_{a-1}$.

Hence we have 
\begin{align*}
    D_{k,l}(\Omega)=&\sum\limits_{i=0}^k \binom{k}{i}\tilde{W}_i(\Omega)-f_k(r_l)\\
    =&\sum\limits_{i=0}^l\binom{k}{i}\left(\tilde{W}_i(\Omega)-\tilde{f}_i(r_l)\right)+\sum\limits_{i=l+1}^k\binom{k}{i}\left(\tilde{W}_i(\Omega)-\tilde{f}_i(r_l)\right)\\
    \geq&\sum\limits_{i=0}^l\binom{k}{i}\left(\tilde{W}_i(\Omega)-\tilde{f}_i(r_l)\right).
\end{align*}

We next prove that \begin{align*}
   0 \leq &\sum\limits_{i=0}^l\binom{k}{i}\left(\tilde{W}_i(\Omega)-\tilde{f}_i(r_l)\right),
\end{align*} 
i.e., 
\begin{align*}
    \sum\limits_{i=0}^{a-1}\binom{k}{i}\left(\tilde{f}_i(r_l)-\tilde{W}_i(\Omega)\right) \leq \sum\limits_{i=a}^l\binom{k}{i}\left(\tilde{W}_i(\Omega)-\tilde{f}_i(r_l)\right)
\end{align*}
Noting that $\frac{\binom{k}{i}}{\binom{l}{i}}$ is strictly increasing with respect to $i$, hence we have 
\begin{align*}
    &\sum\limits_{i=0}^{a-1}\binom{k}{i}\left(\tilde{f}_i(r_l)-\tilde{W}_i(\Omega)\right)\\ \leq& \binom{k}{a-1}\binom{l}{a-1}^{-1} \sum\limits_{i=0}^{a-1}\binom{l}{i}\left(\tilde{f}_i(r_l)-\tilde{W}_i(\Omega)\right)\\=&\binom{k}{a-1}\binom{l}{a-1}^{-1} \sum\limits_{i=a}^l\binom{l}{i}\left(\tilde{W}_i(\Omega)-\tilde{f}_i(r_l)\right) \\\leq& \binom{k}{a}\binom{l}{a}^{-1}\sum\limits_{i=a}^l\binom{l}{i}\left(\tilde{W}_i(\Omega)-\tilde{f}_i(r_l)\right)
    \\ \leq& \sum\limits_{i=a}^l\binom{k}{i}\left(\tilde{W}_i(\Omega)-\tilde{f}_i(r_l)\right).
\end{align*}
Moreover, by the inequality above, we have
\begin{align}
    &D_{k,l}(\Omega)\geq\sum\limits_{i=0}^l\binom{k}{i}\left(\tilde{W}_i(\Omega)-\tilde{f}_i(r_l)\right)\nonumber
    \\\geq& \binom{k}{a}\binom{l}{a}^{-1}\sum\limits_{i=a}^l\binom{l}{i}\left(\tilde{W}_i(\Omega)-\tilde{f}_i(r_l)\right)- \binom{k}{a-1}\binom{l}{a-1}^{-1} \sum\limits_{i=0}^{a-1}\binom{l}{i}\left(\tilde{f}_i(r_l)-\tilde{W}_i(\Omega)\right)\nonumber
    \\\geq&\left(\binom{k}{a}\binom{l}{a}^{-1}-\binom{k}{a-1}\binom{l}{a-1}^{-1} \right) \sum\limits_{i=0}^{a-1}\binom{l}{i}\left(\tilde{f}_i(r_l)-\tilde{W}_i(\Omega)\right)\nonumber\\
    \geq& C(k,l)(\tilde{f}_0(r_l)-\tilde{f}_0(r_0)),
\end{align} 
which completes the proof.
\end{proof}
 Theorem \ref{stability} immediately implies the characterization of equality in the Alexandrov-Fenchel inequality in Theorem \ref{A-F}.

\section{Some properties of smooth modified quermassintegrals preserving shifted curvature flows}\label{flowsection}
In this section we deduce uniform upper bounds for $F$ and $\int_{M_t}|D|d\mu_t$  along the flow \eqref{smoothflow1} depending on the inner radius of initial hypersurfaces. Inspired by \cite{Andrews2013}, we consider the following lemma to prove the weak flow is well defined.    
\begin{lemma} \label{uniformupperboundsonF}
    Suppose that $M=\partial \Omega$ is smooth strictly h-convex. Then along the flow \eqref{smoothflow1} there exists $C=C(\rho_-(\Omega),n,k)$ such that 
   $$F\leq C(1+\frac{1}{\sqrt{t}}), t\in(0,\infty).$$
\end{lemma}
\begin{proof}
For any $t_0\in[0,\infty)$, let $q$ be the center of the inner ball of $\Omega_{t_0}$. Then, by \cite{AndrewsChenWei2021},  there exist $c,\tau>0$ such that for any $t\in[t_0,t_0+\tau)$ there holds $u>2c$, where $u$ is the usual support function of $M_t$. Let $$W=\frac{F}{u-c}.$$

By \cite[Remark 4.1]{Zhang2026}, we have $$c_1(\rho_-(\Omega))\leq\phi(t)\leq c_2(\rho_-(\Omega)).$$
   By direct calculation, we have \begin{align*}
       &C_1(n,k)\leq \dot{F}^{ij}\delta_{ij}\leq C_2(n,k),\\
       &F^2\leq \dot{F}^{ij}(S^2)_{ij}.\end{align*}
    Then, by the calculations in \cite[Proposition 6.2]{AndrewsChenWei2021}, we have 
   \begin{align*}
       \frac{\partial}{\partial t} W\leq& \dot{F}^{ij}(\nabla_i\nabla_j W+\frac{2}{u-c}\nabla_i W\nabla_j u)\\&+W^2(2+F^{-1}\dot{F}^{ij}\delta_{ij})\cosh d_q-c^2W^3\\
       \leq &\dot{F}^{ij}(\nabla_i\nabla_j W+\frac{2}{u-c}\nabla_i W\nabla_j u)+W^2((2+C_2F^{-1})\cosh d_q-c^2W).
   \end{align*}
   Changing the origin repeatedly,  we can adapt Tso's \cite{TSO} technique and the maximum principle  to obtain that 
$$F\leq C(\rho_{-}(\Omega),n,k)(1+\frac{1}{\sqrt{t}}), t\in(0,\infty).$$
\end{proof}
By Lemma \ref{horosupportevolution} and \cite[Corollary 3.2]{LiWanXv2024}, using an  argument similar to that in \cite{Andrews2013}, we immediately obtain the following corollary.
\begin{corollary}\label{yanhuajuli}
    Suppose that $M=\partial \Omega$ is smooth strictly h-convex. Then along the flow \eqref{smoothflow1} there exists $C=C(\rho_-(\Omega),n,k)$ such that
    \begin{align*}
    d(M,M_t)\leq C (t+\sqrt{t}).      
    \end{align*}
\end{corollary}

 We now consider the continuity of the curvature integrals.
\begin{lemma}
  Suppose that $M_1=\partial\Omega_1, M_2=\partial\Omega_2$ are closed h-convex hypersurfaces and $M_1,M_2\subseteq B(\rho)$, where $\rho>0$ and $B(\rho)$ is the geodesic ball of radius $\rho$ centered at the origin. Then we have 
  \begin{align*}
      \left|\int_{\mathcal{N}_{\Omega_1}}g(v)\tilde{E}_ld\mathcal{L}^nv-\int_{\mathcal{N}_{\Omega_2}}g(v)\tilde{E}_ld\mathcal{L}^nv\right|\leq C(\rho,n,l)d_{\mathbb{H}^{n+1}}(M_1,M_2).
  \end{align*}
   
\end{lemma}
\begin{proof}
      By the continuity of curvature integrals, we can assume that $M_1, M_2$ are smooth.
      By the statement in \cite{LiXu2022} (\cite[Lemma 3.4]{Zhang2026}), we have 
      $$  \left|\int_{M_1}\tilde{E}_ld\mu-\int_{M_2}\tilde{E}_ld\mu\right|=\left|(n-2l)(\tilde{W}_l(\Omega_1)-\tilde{W}_{l}(\Omega_2))+(n-l)(\tilde{W}_{l+1}(\Omega_1)-\tilde{W}_{l+1}(\Omega_2))\right|.$$
      By definition we have $|W_{m}(\Omega_1)-W_{m}(\Omega_2)|\leq C(\rho,n,m)d_{\mathbb{H}^{n+1}}(\Omega_1,\Omega_2)$, 
      from which we complete our proof.
\end{proof}
By the argument in \cite[Remark 4.1]{Zhang2026}, for any compact h-convex hypersurface $M=\partial \Omega$, we have $0<c_1(\rho_-(\Omega))<\tilde{\Phi}_{n-l}(\Omega)<c_2(\rho_-(\Omega))$ for all $1\leq l\leq n$.
Hence, by the lemma above, we immediately have the following lemma.
\begin{lemma}\label{estimateglobal}
    Suppose that $M_1=\partial\Omega_1, M_2=\partial\Omega_2$ are closed h-convex hypersurfaces  with $0<r_1\leq\rho_-(\Omega_1),\rho_-(\Omega_2)\leq r_2$. Let $\phi_i=\frac{k+1}{n-k}\frac{\tilde{\Phi}_{n-k-1}(\Omega_i)}{\tilde{\Phi}_{n-k}(\Omega_i)}$. Then we have 
    \begin{align}
        |\phi_1-\phi_2|\leq C(n,k,r_1,r_2)d_{\mathbb{H}^{n+1}}(\Omega_1,\Omega_2).
    \end{align}
\end{lemma}

Inspired by \cite{BCCN2009, Andrews2013}, we next consider the  stability with respect to Hausdorff distance of initial hypersurfaces for the flow \eqref{smoothflow1}.
\begin{lemma}\label{stabilityofflow1}
    For any $r_2> r_1>0$, there exists $T(r_1,r_2)>0$ such that, for any $\varepsilon>0$, there exists $\delta(r_1,r_2,\varepsilon)$  such that for any closed smooth strictly h-convex hypersurfaces $M_1=\partial \Omega_1, M_2=\partial\Omega_2$ satisfying $d_{\mathbb{H}^{n+1}}(M_1,M_2)\leq \delta(r_1,r_2,\varepsilon)$, $r_1\leq \rho_-{(\Omega_1)}\leq\rho_+(\Omega_1)\leq r_2$ and $r_1\leq \rho_-{(\Omega_2)}\leq\rho_+(\Omega_2)\leq r_2$,  we have 
    \begin{align*}
     d_{\mathbb{H}^{n+1}}(M_1(t),M_2(t))\leq \varepsilon,  \quad \text{for all $t\in [0,T_0)$}.
    \end{align*}
    Here $M_1(t), M_2(t)$ are  the solutions of \eqref{smoothflow1} with initial hypersurfaces $M_1$ and $M_2$ respectively. 
\end{lemma}
\begin{proof}
For any $\varepsilon\in(0,1)$,
    let $M_1=\partial \Omega_1$, $M_2=\partial \Omega_2$ be  closed smooth strictly h-convex hypersurfaces such that $r_1\leq \rho_-{(\Omega_i)}\leq\rho_+(\Omega_i)\leq r_2$, $i=1,2$ and $d_{\mathbb{H}^{n+1}}(M_1,M_2)\leq \delta_1(r_1,r_2)<\frac{r_1}{4}$. Choose the center of the inner ball of $\Omega_1$ as the origin. Let $u_i$ be the horosupport function of $M_i$ and $\varphi_i=e^{u_i}$.  Now, let $M_1,M_2$ evolve along the flow \eqref{smoothflow1}.  By arguments in \cite{AndrewsChenWei2021}, there exists $T_1(r_1)<1$ such that 
    $e^{\frac{r_1}{2}}\leq \varphi_1,\varphi_2\leq e^{2r_2}, t\in[0,2T_1)$.
     Let 
    $$\varphi_{i,\lambda}=\lambda\varphi_i(x,\frac{t}{\lambda^2})+1-\lambda,\quad 0<\lambda<1, \quad i=1,2.$$
    By a direct calculation we have $A[\varphi_{i,\lambda}]\geq \lambda A[\varphi_i]$, $i=1,2$. Moreover, by Lemma \ref{horosupportevolution}, we have
   \begin{align}\label{evoofphilambda}
   \frac{\partial}{\partial t}\varphi_{i,\lambda}(x,t)&=\frac{1}{\lambda} (\varphi_i(x,\frac{t}{\lambda^2})\phi_i(\frac{t}{\lambda^2})-F(A[\varphi_i]^{-1}))\nonumber\\
   &\leq \frac{1}{\lambda^2}\varphi_{i,\lambda}(x,t)\phi_i(\frac{t}{\lambda^2})-F(A[\varphi_{i,\lambda}]^{-1}),\quad i=1,2.
   \end{align}
Let 
\begin{align*}d_1(x,t)&=\log\varphi_1-\log\varphi_{2,\lambda},\\
    d_2(x,t)&=\log\varphi_2-\log\varphi_{1,\lambda}.
\end{align*}  Suppose that $$d_i(x_i',t)=\min\limits_{x\in\mathbb{S}^n} d_i(x,t), x'\in\mathbb{S}^n,\quad i=1,2.$$ At $(x_i',t)$ we have 
$\overline\nabla d_i(x_i',t)=0$, $\overline{\nabla}^2d_i(x_i',t)\geq 0$. Noticing that \begin{align*}\log\varphi_1(x,t)&=d_1(x,t)+\log\varphi_{2,\lambda}(x,t),\\
\log \varphi_2(x,t)&=d_2(x,t)+\log\varphi_{1,\lambda}(x,t),\end{align*} by a direct calculation we have that if $d_i(x,t)>0$ there holds
\begin{align*}
    \frac{F(A[\varphi_{2,\lambda}]^{-1})}{\varphi_{2,\lambda}}&\geq \frac{F(A[\varphi_1]^{-1})}{\varphi_1},\\
    \frac{F(A[\varphi_{1,\lambda}]^{-1})}{\varphi_{1,\lambda}}&\geq \frac{F(A[\varphi_2]^{-1})}{\varphi_2}.
\end{align*}
Hence we have that if $d_i(x,t)>0$ there holds
\begin{align*}
\frac{\partial}{\partial t}d_1(x_1',t)&\geq   \phi_1(t)-\frac{1}{\lambda^2}\phi_2(\frac{t}{\lambda^2})\\
\frac{\partial}{\partial t}d_2(x_2',t)&\geq   \phi_2(t)-\frac{1}{\lambda^2}\phi_1(\frac{t}{\lambda^2}).  
\end{align*}
Hence, if $d_i(x,t)>0$ on the time interval $[0,T^*)\subseteq [0,2T_1)$, we have 
\begin{align*}
    \min\limits_{x\in\mathbb{S}^n}\{d_1(x,t)\}\geq& \min\limits_{x\in\mathbb{S}^n}\{d_1(x,0)\}+\int_0^t \left(\phi_1(s)-\frac{1}{\lambda^2}\phi_2(\frac{s}{\lambda^2})\right)ds\nonumber\\
    =& \min\limits_{x\in\mathbb{S}^n}\{\log \varphi_1-\log\varphi_2-\log(1-(1-\lambda)(1-\varphi_2^{-1}))\}\nonumber\\
    &+\int_0^{t} \left(\phi_1(s)-\phi_2(s)\right)ds-\int_t^{\lambda^{-2}t}\phi_2(s)ds\nonumber\\
    \geq& (1-\lambda)(1-e^{-\frac{r_1}{2}})-\delta_1-\left|\int_0^{t} \left(\phi_1(s)-\phi_2(s)\right)ds\right|\nonumber\\
    &+M(r_1,r_2)t(1-\lambda^{-2}), t\in[0,T^*),
\end{align*}
where $M(r_1,r_2)$ is a constant such that $2\phi_2(t)\leq M(r_1,r_2)$ for any $t\in[0,2T_1)$. Noticing that $\varphi_{2,\lambda}$ is increasing with respect to $\lambda$, we can choose $\lambda(\delta_1)>0$ such that $\varphi_1(x,0)> \varphi_{2,\lambda}(x,0)$ and $$\lambda<1-\frac{2\delta_1}{1-e^{-\frac{r_1}{2}}}.$$ Then we have 
\begin{align}
    \min\limits_{x\in\mathbb{S}^n}\{d_1(x,t)\}
    \geq& \frac{1-\lambda}{2}(1-e^{-\frac{r_1}{2}})-\left|\int_0^{t} \left(\phi_1(s)-\phi_2(s)\right)ds\right|\nonumber\\
    &+M(r_1,r_2)t(1-\lambda^{-2}), t\in[0,T^*).\label{estimated}
\end{align}

Then, by Corollary \ref{yanhuajuli}, for any $t\in[0,T^*]$ there holds
\begin{align*}
    &\log\varphi_2(x,t)-\log\varphi_1(x,t)\nonumber\\=&\log\varphi_2(x,t)-\log(\lambda\varphi_2(x,t)
    +1-\lambda)\nonumber\\
    &+\log(\lambda\varphi_2(x,t)+1-\lambda)-\log\varphi_{2,\lambda}(x,t)\nonumber\\
    &+\log\varphi_{2,\lambda}(x,t)-\log\varphi_1(x,t)\nonumber\\
    \leq&\log\left(1+\frac{(1-\lambda)\varphi_2(x,t)}{\lambda\varphi_2(x,t)
    +1-\lambda}\right)+\log\left(1+\frac{\lambda(\varphi_2(x,t)-\varphi_2(x,\frac{t}{\lambda^2}))}{\lambda\varphi_2(x,\frac{t}{\lambda^2})+1-\lambda}\right)\nonumber
\end{align*}
By Lemma \ref{uniformupperboundsonF} and Lemma \ref{horosupportevolution}, we have
$$\varphi_2(x,t)-\varphi_2(x,\frac{t}{\lambda^2})\leq C(r_1,r_2)(t+\sqrt{t})(1-\lambda).$$
Hence there exists $C'(r_1,r_2)>1$ such that 
\begin{align}\label{zongguji}
    \log\varphi_2(x,t)-\log\varphi_1(x,t)
    \leq C'(r_1,r_2)(1-\lambda)(1+\sqrt{t}+t).
\end{align}
  By an argument similar to that above, without loss of generality we assume that 
\begin{align*}
    \log\varphi_1(x,t)-\log\varphi_2(x,t)\nonumber
    \leq C'(r_1,r_2)(1-\lambda)(1+\sqrt{t}+t).
\end{align*}

Then by Lemma \ref{estimateglobal} we have 
\begin{align*}
|\phi_1(t)-\phi_2(t)|\leq     C(n,k,r_1,r_2)(1-\lambda), \forall t\in[0,T^*].
\end{align*}
Hence, by \eqref{estimated}, we have 
\begin{align}
    \min\limits_{x\in\mathbb{S}^n}\{d_1(x,t)\}
    \geq& \frac{1-\lambda}{2}(1-e^{-\frac{r_1}{2}})-C(n,k,r_1,r_2)(1-\lambda)t\nonumber\\
    &-M(r_1,r_2)\frac{1+\lambda}{\lambda^2}(1-\lambda)t, t\in[0,T^*).\label{estimateofd2}
\end{align}
We remark that if $\delta_1$ is close enough to $0$ we can choose $\lambda(\delta_1)$ close enough to $1$. We can fix a small $\delta_0(r_1,r_2,\varepsilon)$ such that if $0\leq\delta_1\leq\delta_0$ we can choose $\lambda$ such that 
\begin{align}
    &\varphi_1(x,0)> \varphi_{2,\lambda}(x,0),\nonumber\\
    &\lambda<1-\frac{2\delta_1}{1-e^{-\frac{r_1}{2}}},\nonumber\\
    &\frac{1}{2}\leq\lambda<1,\nonumber\\
    &1-\frac{\varepsilon}{5C'(r_1,r_2)}<\lambda\label{lambdatiaojian}.
\end{align}
Let $$T_0=\min\left\{\frac{1-e^{\frac{-r_1}{2}}}{2(C(n,k,r_1,r_2)+10M(r_1,r_2)},\frac{T_1}{2}\right\}.$$ 
Then, for any $t\in[0,\min\{T_0,T^*\})$, by \eqref{estimateofd2}, we have 
$$\min\limits_{x\in\mathbb{S}^n}\{d_1(x,t)\}>C''(n,k,r_1,r_2)(1-\lambda)>0.$$ Hence we have $$\min\limits_{x\in\mathbb{S}^n}\{d_1(x,t)\}>0, \forall t\in[0,T_0).$$
Then, by  \eqref{zongguji} and \eqref{lambdatiaojian}, we have 
$$\log\varphi_2(x,t)-\log\varphi_1(x,t)<\varepsilon, \forall \delta_1\in[0,\delta_0), t\in[0,T_0). $$
We can apply similar arguments to prove that there exist $\delta'(r_1,r_2,\varepsilon)$ and $T'_0(r_1,r_2)$ such that if $d_{\mathbb{H}^{n+1}}(M_1,M_2)\leq \delta'(r_1,r_2,\varepsilon)$ there holds \begin{align*}
    \log \varphi_1-\log\varphi_2< \varepsilon, \forall t\in[0,T_0').
\end{align*}
Then, by \cite{LiWanXv2024}, we complete our proof.
\end{proof}
By applying Lemma \ref{stabilityofflow1} successively on time intervals, we obtain the following lemma.
\begin{lemma}\label{stabilityofflow}
    Suppose that $\Omega\in\mathcal{H}(\mathbb{H}^{n+1})$. Let $\{\Omega_i\}$ be a sequence of smooth strictly h-convex bodies such that $\Omega_i\rightarrow\Omega$ as $i\rightarrow\infty$. Evolve $M_j=\partial\Omega_j$ along the flow \eqref{smoothflow1} and denote by $M_{j,t}=\partial\Omega_{j,t}$ the evolving hypersurface of $M_j$ at time $t\in[0,\infty)$. Then,  for any $t\in[0,\infty)$, there holds 
    $$d_{\mathbb{H}^{n+1}}(M_{j,t},M_{i,t})\rightarrow 0, i,j\rightarrow\infty.$$
\end{lemma}
Suppose that $\Omega$ is smooth strictly h-convex and $M_t$ is the solution to \eqref{smoothflow1} with initial hypersurface $M=\partial\Omega$.
We next prove that  $\int_{M_t}|D|d\mu_t$ can be controlled by $t$ and $\int_M|D|d\mu$.

\begin{lemma}\label{lemmaupperboundsonintD}
Suppose that $\Omega$ is smooth strictly h-convex and $M_t$ is the solution to \eqref{smoothflow1} with initial hypersurface $M=\partial\Omega$. Then
there exists $C=C(\rho_-(\Omega),n,k)$ such that
    \begin{align*}
        \int_{M_t} |D|d\mu_t\leq \exp(C(\sqrt{t}+t))\int_{M_0} |D|d\mu.
    \end{align*}
\end{lemma}
\begin{proof}
    By Lemma \ref{intDevo}, we have \begin{align*}
    &\frac{\partial}{\partial t} \int_{M_t}|D| d\mu_t\nonumber\\
    =&\int_{M_t}\mathrm{sgn}(D)\dot{D}^{ij} \nabla_i\nabla_j Fd\mu_t\nonumber\\
    &+\int_{M_t}\frac{\mathrm{sgn}(D)\int_{M_t}\tilde{E}_kd\mu_t-\int_{M_t}\mathrm{sgn}(D)\tilde{E}_kd\mu_t}{\int_{M_t}\tilde{E}_k d\mu_t}\\&((F-\phi(t))\dot{D}^{ij}((S^2)^{i}_j+2S^i_j)+(\phi(t)-F)HD)d\mu_t.
    \end{align*}
We first estimate $\int_{M_t}\mathrm{sgn}(D)\dot{D}^{ij} \nabla_i\nabla_j Fd\mu_t$. Let $M_{+,t,\varepsilon}(t):=\{x:F-\phi(t)>\varepsilon\}$ and $M_{-,t,\varepsilon}:=\{x:F-\phi(t)<-\varepsilon\}$. We choose a set $\{\varepsilon_m\}$ of regular values of $F-\phi(t)$ and $\varepsilon_m$ tends to $0$ as $m$ tends to $\infty$. 
If $F=\phi(t)$ there holds $\dot{D}^{ij}=\tilde{E}_k\dot{F}^{ij}>0$. 
 Then we have \begin{align*}
    \int_{M_{+,t,\varepsilon_m}(t)} \mathrm{sgn}(D)\dot{D}^{ij} \nabla_i\nabla_j Fd\mu_t=-\int_{\{F-\phi(t)=\varepsilon_m\}}\frac{\dot{D}^{ij}\nabla_iF\nabla_jF}{|\nabla F|} d\mu_t\leq 0
\end{align*} for $m$ large enough.
Hence we have $$\int_{\{F-\phi(t)>0\}}\mathrm{sgn}(D)\dot{D}^{ij} \nabla_i\nabla_j Fd\mu_t\leq 0.$$ A similar argument yields $$\int_{\{F-\phi(t)<0\}}\mathrm{sgn}(D)\dot{D}^{ij} \nabla_i\nabla_j Fd\mu_t\leq 0.$$
Hence we have 
\begin{align*}
    &\frac{\partial}{\partial t} \int_{M_t}|D| d\mu_t\nonumber\\
    \leq&
    \int_{M_t}\frac{\mathrm{sgn}(D)\int_{M_t}\tilde{E}_kd\mu_t-\int_{M_t}\mathrm{sgn}(D)\tilde{E}_kd\mu_t}{\int_{M_t}\tilde{E}_k d\mu_t}((F-\phi(t))\dot{D}^{ij}((S^2)^{i}_j+2S^i_j)\\&+(\phi(t)-F)HD)d\mu_t\nonumber\\
    =&\int_{M_t}\frac{\mathrm{sgn}(D)\int_{M_t}\tilde{E}_kd\mu_t-\int_{M_t}\mathrm{sgn}(D)\tilde{E}_kd\mu_t}{\int_{M_t}\tilde{E}_k d\mu_t}(F-\phi(t))(\dot{D}^{ij}((S^2)^{i}_j+2S^i_j)-HD)d\mu_t\nonumber\\
    =&\int_{M_t}\frac{\mathrm{sgn}(D)\int_{M_t}\tilde{E}_kd\mu_t-\int_{M_t}\mathrm{sgn}(D)\tilde{E}_kd\mu_t}{\int_{M_t}\tilde{E}_k d\mu_t}D\frac{1}{\tilde{E}_k}(\dot{D}^{ij}((S^2)^{i}_j+2S^i_j)-HD)d\mu_t\nonumber.
    \end{align*}
Noticing that  $\dot{D}^{ij}=\dot{\tilde{E}}_{k+1}^{ij}-\phi(t)\dot{\tilde{E}}_{k}^{ij}$, hence, by direct calculation, we have 
\begin{align*}
    &\frac{1}{\tilde{E}_k}(\dot{D}^{ij}((S^2)^{i}_j+2S^i_j)-HD)\\=&\frac{1}{\tilde{E}_k}((2k+2-n+(n-k)\phi(t))\tilde{E}_{k+1}+(n-2k)\phi(t)\tilde{E}_k-(n-k-1)\tilde{E}_{k+2})\\
    \leq &C_1(\rho_-(\Omega),n,k)(F+1).
\end{align*}
Noticing that $$-1\leq\frac{\int_{M_t}\mathrm{sgn}(D)\tilde{E}_kd\mu_t}{\int_{M_t}\tilde{E}_k d\mu_t}\leq 1,$$ we have 
$$\frac{\mathrm{sgn}(D)\int_{M_t}\tilde{E}_kd\mu_t-\int_{M_t}\mathrm{sgn}(D)\tilde{E}_kd\mu_t}{\int_{M_t}\tilde{E}_k d\mu_t}D\geq 0.$$
Hence we have 
\begin{align*}
    \frac{\partial}{\partial t} \int_{M_t}|D|d\mu_t\leq \nonumber\int_{M_t}\frac{\mathrm{sgn}(D)\int_{M_t}\tilde{E}_kd\mu_t-\int_{M_t}\mathrm{sgn}(D)\tilde{E}_kd\mu_t}{\int_{M_t}\tilde{E}_k d\mu_t}DC_1(\rho_-(\Omega),n,k)(F+1)d\mu_t. 
\end{align*}
Hence, by Lemma \ref{uniformupperboundsonF}, we have 
\begin{align*}
    \frac{\partial}{\partial t} \int_{M_t}|D|d\mu_t\leq& C(\rho_-(\Omega),n,k)(1+\frac{1}{\sqrt{t}})\int_{M_t}\frac{\mathrm{sgn}(D)\int_{M_t}\tilde{E}_kd\mu_t-\int_{M_t}\mathrm{sgn}(D)\tilde{E}_kd\mu_t}{\int_{M_t}\tilde{E}_k d\mu_t}Dd\mu_t\\
    =&C(\rho_-(\Omega),n,k)(1+\frac{1}{\sqrt{t}})\int_{M_t}|D|d\mu_t, 
\end{align*}
from which we immediately complete the proof.
\end{proof}

\section{Characterization of geodesic balls via modified support measures}\label{chraofgeoball}
As introduced in \cite{Kohlmann1991, Federer1959}, a subset $\Omega\subseteq \mathbb{H}^{n+1}$ is called of positive reach if there exists $r>0$ such that for any $p\in \Omega^r$ there exists a  unique point $x\in\Omega$ such that $d_{\mathbb{H}^{n+1}}(p,\Omega)=d_{\mathbb{H}^{n+1}}(p,x)$. Moreover, for any $0<\varepsilon<r$, Walter \cite{Walter74, Walter76} proved that $\partial \Omega^\varepsilon$ is of class $C^{1,1}$.
We first introduce a characterization for h-convex bodies.
\begin{lemma}
    Suppose that $\Omega$ is a compact domain of positive reach. Then the following statements are equivalent: 

    (1)$k_i(v)\geq 1$, $\forall v\in \tilde{\mathcal{N}}_\Omega$.

    (2) $\Omega\in\mathcal{H}(\mathbb{H}^{n+1})$.
\end{lemma}
\begin{proof}
    (2)$\implies $(1): See \cite{Zhang2026}.

    (1)$\implies$(2): By the Blaschke selection type Theorem \ref{hausdorffofhconvex} and Lemma \ref{outerhconvex}, we can assume that $M=\partial \Omega$ is of class $C^{1,1}$ without loss of generality. For any $p\in M$, we will prove that there exists a horosphere that supports $\Omega$ at $p$. Denote by  $v(p)$ the unit outer normal vector of $M$ at $p$. Denote by $\xi_p$ the ideal point of the direction of $-v(p)$. Let $\overline{\mathbb{R}}^{n+1}_+$ be the half plane model of $\mathbb{H}^{n+1}$ in which $\xi_p$ corresponds to $\infty$. Then the horosphere $S$  passing through $p$ with center $\xi_p$ is a Euclidean plane. For any Alexandrov differentiable point $p'$ of $M$, by the calculation in \cite{Guan2009}, we know that the Euclidean second fundamental form of $M$ at $p'$ is positive semidefinite. Hence, by the statement in  \cite{Hug2022}, we conclude $\Omega$ is a Euclidean convex body, from which we have  that $S$ supports $\Omega$ at $p$. Hence we have that $\Omega$ is h-convex.
\end{proof}
Now we consider the rigidity theorem for modified support measures. We first consider the rigidity theorem for modified area measures.

\begin{proof}[Proof of Theorem \ref{arearigidity}]
    Suppose that $\Omega\in\mathcal{H}(\mathbb{H}^{n+1})$ and  
\begin{align*}
     \tilde S_{0}(\Omega,\cdot)=c\tilde{S}_{1}(\Omega,\cdot),
\end{align*}
Choose an interior point of $\Omega$ as the origin.
By \eqref{areameasure}, we have 
\begin{align*}
    \mathrm{tr} A[\varphi_\Omega]=\frac{1}{c\varphi_\Omega}d\mathcal{L}^n,
\end{align*}
from which we have 
\begin{align*}
    [\Delta \varphi_\Omega]=\left(\frac{1}{c\varphi_\Omega}+n\left(\frac{|\overline\nabla\varphi_\Omega|^2}{2\varphi_\Omega}-\frac{\varphi_\Omega-\varphi_\Omega^{-1}}{2}\right)\right)d\mathcal{L}^n.
\end{align*}
Then, by the semi-convexity of $\varphi_\Omega$ we have $$\frac{1}{c\varphi_\Omega}+n\left(\frac{|\overline\nabla\varphi_\Omega|^2}{2\varphi_\Omega}-\frac{\varphi_\Omega-\varphi_\Omega^{-1}}{2}\right)\in\mathcal{L}^\infty(\mathbb{S}^n),$$ from which we have that $\varphi_\Omega$ is of class $C^{1,1}$. Hence we have \begin{align*}
    \varphi_\Omega\Delta \varphi_\Omega=\frac{1}{c}+n\left(\frac{|\overline\nabla\varphi_\Omega|^2}{2}-\frac{\varphi_\Omega^2-1}{2}\right).
\end{align*}
By standard elliptic regularity and Schauder bootstrapping, we have that $\varphi_\Omega$ is of class $C^\infty$.
Now, let $\varepsilon>0$ and $\varphi^\varepsilon_\Omega=e^\varepsilon\varphi_\Omega$. Then, by direct calculation, we have 
\begin{align*}
A[\varphi^\varepsilon_\Omega]=e^\varepsilon A[\varphi_\Omega]+\frac{e^\varepsilon-e^{-\varepsilon}}{2\varphi_\Omega}g>0.    
\end{align*} 
Moreover, we have 
\begin{align*}
\mathrm{tr} A[\varphi^\varepsilon_\Omega]=\frac{e^{2\varepsilon}}{c\varphi_\Omega^\varepsilon}+n\frac{e^{2\varepsilon}-1}{2\varphi_\Omega^\varepsilon}.    
\end{align*}
Then, by \cite[Theorem 8.1 (7)]{LiXu2022} and \cite[Proposition 3.1]{LiWanXv2024}, we have that $\Omega^\varepsilon$ is a geodesic ball, from which we complete our proof.
\end{proof}

Having established the rigidity theorem above, we now consider the rigidity theorem for modified support measures. 

\begin{proof}[Proof of Theorem \ref{metricball}]

Let $\Omega\in\mathcal{H}(\mathbb{H}^{n+1})$ satisfy 
\begin{align*}
     \tilde\Phi_{n-k-1}(\Omega,\cdot)=c\tilde{\Phi}_{n-k}(\Omega,\cdot).   
\end{align*}
Assume that $k< n-1$. We will prove that $$\tilde{\Phi}_{n-k-1}(\Omega,\cdot)=c\tilde{\Phi}_{n-k}(\Omega,\cdot)\implies\tilde{\Phi}_{n-k-2}(\Omega,\cdot)=\frac{(k+1)(n-k-1)}{(k+2)(n-k)}c\tilde{\Phi}_{n-k-1}.$$
Now, suppose that  $$\tilde{\Phi}_{n-k-1}(\Omega,\cdot)=c\tilde{\Phi}_{n-k}(\Omega,\cdot).$$ 
Then we have 
\begin{align}\label{tiaojian3}
    \int_{\mathcal{N}_\Omega}|g_\Omega(v)\tilde{E}_{k+1}(v)-c'g_{\Omega}(v)\tilde{E}_k(v)|\mathcal{L}^n v=0,
\end{align}
here $c'=\binom{n}{k}\binom{n}{k+1}^{-1}c=\frac{k+1}{n-k}c$. Actually, we have 
$$c'=\frac{\int_{\mathcal{N}_\Omega}g_\Omega(v)\tilde{E}_{k+1}(v)\mathcal{L}^n v}{\int_{\mathcal{N}_\Omega}g_{\Omega}(v)\tilde{E}_k(v)\mathcal{L}^n v}$$
We first prove that there exists a sequence of smooth strictly h-convex bodies $\{\Omega_j\}_{j\in\mathbb{N}}$ such that $\Omega_j\rightarrow\Omega$ as $j$ tends to $\infty$ and 
\begin{align*}
\int_{\partial\Omega_j} |\tilde{E}_{k+1}-c'\tilde{E}_k |d\mu_j\rightarrow 0, j\rightarrow\infty.    
\end{align*}

We consider the outer parallel set $\Omega^\varepsilon$ of $\Omega$, $\varepsilon<1$.
For any $v\in\mathcal{N}_\Omega$, let $f_\varepsilon(v)=\Phi_{\Pi(v)}(\varepsilon v)$, where $\Phi$ is the geodesic flow. Then $f_\varepsilon$ is a bi-Lipschitz map from $\mathcal{N}_\Omega$ to $\mathcal{N}_{\Omega^\varepsilon}$. 
By calculations in \cite{Kohlmann1991}, we have 
$$k_i(f_\varepsilon(v))-1=\frac{\sinh \varepsilon +\cosh\varepsilon k_i(v)}{\cosh\varepsilon+\sinh\varepsilon k_i(v)}-1=\frac{e^{-\varepsilon}(k_i(v)-1)}{\cosh\varepsilon+\sinh\varepsilon k_i(v)},$$
$$g_{\Omega^\varepsilon}(f_\varepsilon(v))=\prod\limits_{i=1}\limits^n \frac{\cosh \varepsilon+\sinh\varepsilon k_i(v)}{\sqrt{(\cosh \varepsilon+\sinh \varepsilon k_i(v))^2+(\sinh\varepsilon+\cosh\varepsilon k_i(v))^2}}$$ 
 and
 $$J_n(f_\varepsilon)(v)=\prod\limits_{i=1}\limits^n \frac{\sqrt{(\cosh \varepsilon+\sinh \varepsilon k_i(v))^2+(\sinh\varepsilon+\cosh\varepsilon k_i(v))^2}}{\sqrt{1+k_i^2(v)}},$$
 here, if $k_i=\infty$, the calculation accords with the algebraic conventions in Section \ref{preli}.   
Hence, by direct calculation, we have 
\begin{align*}
&\binom{n}{l}J_n(f_\varepsilon)(v)g_{\Omega^{\varepsilon}}(f_\varepsilon(v))\tilde{E}_l(f_\varepsilon(v))\\=&e^{-l\varepsilon}\prod_{i=1}^n \frac{\cosh\varepsilon+\sinh\varepsilon k_i(v)}{\sqrt{1+k_i^2(v)}}\sum_{|I|=l,I\subseteq\{1,\dots,n\}}\prod_{i\in I}\frac{k_i(v)-1}{\cosh\varepsilon+\sinh\varepsilon k_i(v)}\\
=&e^{-l\varepsilon}g_{\Omega}(v)\sum_{|I|=l,I\subseteq\{1,\dots,n\}}\prod_{i\notin I}^n (\cosh\varepsilon+\sinh\varepsilon k_i(v))\prod_{i\in I}(k_i(v)-1)\\
=&e^{-l\varepsilon}g_{\Omega}(v)\sum_{|I|=l,I\subseteq\{1,\dots,n\}}\prod_{i\notin I}^n (e^\varepsilon+\sinh\varepsilon (k_i(v)-1))\prod_{i\in I}(k_i(v)-1)\\
=&e^{-l\varepsilon}g_\Omega(v)\sum_{m=0}^{n-l}e^{(n-l-m)\varepsilon}\sinh^m\varepsilon\sum_{|J|=l+m,J\subseteq\{1,\dots,n\}}\sum_{|I|=l,I\subseteq J}\prod_{i\in J}(k_i(v)-1)\\
=&e^{-l\varepsilon}g_\Omega(v)\sum_{m=0}^{n-l}e^{(n-l-m)\varepsilon}\sinh^m\varepsilon\binom{l+m}{l}\sum_{|J|=l+m,J\subseteq\{1,\dots,n\}}\prod_{i\in J}(k_i(v)-1)\\
=&e^{-l\varepsilon}g_\Omega(v)\sum_{m=0}^{n-l}e^{(n-l-m)\varepsilon}\sinh^m\varepsilon\binom{l+m}{l}\binom{n}{l+m}\tilde{E}_{l+m}(v).
\end{align*}
Hence we have 
\begin{align*}
    J_n(f_\varepsilon)(v)g_{\Omega^{\varepsilon}}(f_\varepsilon(v))\tilde{E}_l(f_\varepsilon(v))=\sum_{m=0}^{n-l}\binom{n-l}{m}e^{(n-2l-m)\varepsilon}\sinh^m\varepsilon g_\Omega(v)\tilde{E}_{l+m}(v),
\end{align*}
which implies 
\begin{align*}
    &\int_{\mathcal{N}_\Omega}|J_n(f_\varepsilon)(v)g_{\Omega^{\varepsilon}}(f_\varepsilon(v))\tilde{E}_l(f_\varepsilon(v))-g(v)\tilde{E}_l(v)|d\mathcal{L}^n v\\
    \leq&\int_{\mathcal{N}_\Omega}|e^{(n-2l)\varepsilon}-1| g_\Omega(v)\tilde{E}_{l}(v)d\mathcal{L}^n v\\&+\int_{\mathcal{N}_\Omega}\sum_{m=1}^{n-l}\binom{n-l}{m}e^{(n-2l-m)\varepsilon}\sinh^m\varepsilon g_\Omega(v)\tilde{E}_{l+m}(v)d\mathcal{L}^n v\\
    \leq&C'(n,\rho_-(\Omega))(|e^{(n-2l)\varepsilon}-1|+\sinh \varepsilon).
\end{align*}
 By \eqref{tiaojian3} and the area formula \cite{Federer}, we have 
\begin{align*}
    &\int_{\partial\Omega^\varepsilon} |\tilde{E}_{k+1}-c'\tilde{E}_k|d\mu\\
    =&\int_{\mathcal{N}_{\Omega^\varepsilon}}g_{\Omega^\varepsilon}(v)|\tilde{E}_{k+1}-c'\tilde{E}_k|(v)\mathcal{L}^n v\\
    =&\int_{\mathcal{N}_{\Omega}}g_{\Omega^\varepsilon}(f_\varepsilon(v))J_n(f_\varepsilon)(v)|\tilde{E}_{k+1}-c'\tilde{E}_k|(f_\varepsilon(v))\mathcal{L}^n v\\
    \leq&\int_{\mathcal{N}_{\Omega}}|g_{\Omega^\varepsilon}(f_\varepsilon(v))J_n(f_\varepsilon)(v)\tilde{E}_{k+1}(f_\varepsilon(v))-g_\Omega(v)\tilde{E}_{k+1}(v)|\mathcal{L}^n v\\
    &+\int_{\mathcal{N}_{\Omega}}|g_{\Omega}(v)\tilde{E}_{k+1}(v)-c'g_\Omega(v)\tilde{E}_{k}(v)|\mathcal{L}^n v\\
    &+c'\int_{\mathcal{N}_{\Omega}}|g_\Omega(v)\tilde{E}_{k}(v)-g_{\Omega^\varepsilon}(f_\varepsilon(v))J_n(f_\varepsilon)(v)\tilde{E}_{k}|\mathcal{L}^n v\\
    =&\int_{\mathcal{N}_{\Omega}}|g_{\Omega^\varepsilon}(f_\varepsilon(v))J_n(f_\varepsilon)(v)\tilde{E}_{k+1}(f_\varepsilon(v))-g_\Omega(v)\tilde{E}_{k+1}(v)|\mathcal{L}^n v\\
    &+c'\int_{\mathcal{N}_{\Omega}}|g_\Omega(v)\tilde{E}_{k}(v)-g_{\Omega^\varepsilon}(f_\varepsilon(v))J_n(f_\varepsilon)(v)\tilde{E}_{k}|\mathcal{L}^n v\\
    \leq &C(n,\rho_-(\Omega))(|e^{(n-2k)\varepsilon}-1|+|e^{(n-2k-2)\varepsilon}-1|+\sinh \varepsilon)\rightarrow 0, \quad\text{as}\quad \varepsilon\rightarrow 0.
\end{align*}
For any $\varepsilon>0$, noticing that $\partial\Omega^\varepsilon$ is of class $C^{1,1}$, by Lemma \ref{higherregularapproximation}, there exists a sequence of smooth strictly h-convex bodies $\{\Omega^\varepsilon_i\}$ such that $\Omega_i^\varepsilon\rightarrow\Omega^\varepsilon$ in the sense of $C^{1,\alpha}\cap W^{2,k+1}$. Then we have 
\begin{align*}
    \int_{\partial\Omega_i^\varepsilon}|\tilde{E}_{k+1}-c'\tilde{E}_k| d\mu\rightarrow\int_{\partial\Omega^\varepsilon}|\tilde{E}_{k+1}-c'\tilde{E}_k| d\mu,\,\text{as} \, i\rightarrow\infty.
\end{align*}
Then,  considering the smooth approximation of $\Omega^\varepsilon$ in the sense of $C^{1,\alpha}\cap W^{2,k+1}$ and letting $\varepsilon$ tend to $0$, 
we have that there exists a sequence of smooth strictly h-convex bodies $\{\Omega_j\}_{j\in\mathbb{N}}$ such that $\Omega_j\rightarrow\Omega$ as $j$ tends to $\infty$ and 
\begin{align*}
\int_{\partial\Omega_j} |\tilde{E}_{k+1}-c'\tilde{E}_k |d\mu_j\rightarrow 0, j\rightarrow\infty.    
\end{align*}
By Theorem \ref{weakconsupport} and the uniform lower bounds on the curvature integrals, we have 
$$\frac{\int_{\partial\Omega_j} \tilde{E}_{k+1}d\mu_j}{\int_{\partial\Omega_j} \tilde{E}_k d\mu_j}\rightarrow c',\, \text{as}\, j\rightarrow\infty,$$
from which we have 
\begin{align*}
    \int_{\partial\Omega_j} \left|\tilde{E}_{k+1}-\frac{\int_{\partial\Omega_j} \tilde{E}_{k+1}d\mu_j}{\int_{\partial\Omega_j} \tilde{E}_k d\mu_j}\tilde{E}_k \right|d\mu_j\rightarrow 0, j\rightarrow\infty.
\end{align*}

Let  $M_{j,t}$ be the solutions of \eqref{smoothflow1} with initial hypersurface $M_j=\partial\Omega_j$, $\phi_j$ be the corresponding global term and $D_{j,t}=\tilde{E}_{k+1}-\phi_j(t)\tilde{E}_k$. In the following we assume that $j$ is large enough such that $d_{\mathbb{H}^{n+1}}(\Omega_j,\Omega)$ is small enough. Let $t\in[0,1)$. By Lemma \ref{lemmaupperboundsonintD}, we have 
$$\int_{M_{j,t}}|D_{j,t}|d\mu_{j,t}\leq e^{2C(\rho_-(\Omega),n,k)}\int_{M_j}|D_{j,0}|d\mu_j.$$
 If $\tilde{E}_{k+2}-\phi_j(t)\tilde{E}_{k+1}\geq0$, by the Newton-Maclaurin inequality and Lemma \ref{uniformupperboundsonF},  we have $$\tilde{E}_{k+2}-\phi_j(t)\tilde{E}_{k+1}\leq \frac{\tilde{E}_{k+1}}{\tilde{E}_k}(\tilde{E}_{k+1}-\phi_j(t)\tilde{E}_{k})\leq C(\rho_-(\Omega),n,k)(1+\frac{1}{{\sqrt t}})(\tilde{E}_{k+1}-\phi_j(t)\tilde{E}_{k}).$$
Choose the center of an inner ball of $\Omega$ as the origin $q$. By Theorem \ref{shiftedminkow}, we have 
\begin{align*}
    &\int_{M_{j,t}}S_q(\tilde{E}_{k+2}-\phi_j(t)\tilde{E}_{k+1})_-d\mu_{j,t}\\
    =&\int_{M_{j,t}}S_q(\tilde{E}_{k+2}-\phi_j(t)\tilde{E}_{k+1})_+d\mu_{j,t}-\int_{M_{j,t}}(\cosh d_q-S_q)(\tilde{E}_{k+1}-\phi_j(t)\tilde{E}_{k})d\mu_{j,t}.
\end{align*}
Then, noticing that $$\int_{M_{j,t}}\tilde{E}_{k+1}-\phi_j(t)\tilde{E}_{k}d\mu_{j,t}=0$$ 
and $\cosh d_q-S_q$, $S_q$ have uniform lower and upper bounds (see the argument in \cite{Kohlmann1998}) depending on $\rho_-(\Omega)$,
 we have 
\begin{align}\label{controll}
    \int_{M_{j,t}}|\tilde{E}_{k+2}-\phi_j(t)\tilde{E}_{k+1}|d\mu_{j,t}\leq C'(\rho_-(\Omega),n,k)(1+\frac{1}{{\sqrt t}})\int_{M_{j,t}}|\tilde{E}_{k+1}-\phi_j(t)\tilde{E}_{k}|d\mu_{j,t}. 
\end{align}

By Corollary \ref{yanhuajuli}, we have $\Omega_{j,t}\rightarrow\Omega_j$ as $t\rightarrow 0$. Then, by Theorem \ref{weakconsupport} and Corollary \ref{yanhuajuli}, we have $\phi_j(t)\rightarrow\phi_j(0)$ uniformly as $t\rightarrow 0$. 
Now, we first fix $t$ small enough such that $|\phi_j(t)-\phi_j(0)|$ is small enough uniformly for all $j\in\mathbb{N}$. 
Next, we choose $j$ large enough such that $|\phi_j(0)-c'|$ is small enough and $$C'(\rho_-(\Omega),n,k)(1+\frac{1}{{\sqrt t}})\int_{M_{j,t}}|\tilde{E}_{k+1}-\phi_j(t)\tilde{E}_{k}|d\mu_{j,t}$$ is small enough.
Notice that  
\begin{align*}
    &\int_{M_{j,t}}|\tilde{E}_{k+2}-c'\tilde{E}_{k+1}|d\mu_{j,t}\\
    \leq&\int_{M_{j,t}}|\tilde{E}_{k+2}-\phi_j(t)\tilde{E}_{k+1}|d\mu_{j,t}+\int_{M_{j,t}}|\phi_j(t)\tilde{E}_{k+1}-\phi_j(0)\tilde{E}_{k+1}|d\mu_{j,t}\\&+\int_{M_{j,t}}|\phi_j(0)\tilde{E}_{k+1}-c'\tilde{E}_{k+1}|d\mu_{j,t}\\
    \leq& C'(\rho_-(\Omega),n,k)(1+\frac{1}{{\sqrt t}})\int_{M_{j,t}}|\tilde{E}_{k+1}-\phi_j(t)\tilde{E}_{k}|d\mu_{j,t}+\int_{M_{j,t}}|\phi_j(t)-\phi_j(0)|\tilde{E}_{k+1}d\mu_{j,t}\\
    &+\int_{M_{j,t}}|\phi_j(0)-c'|\tilde{E}_{k+1}d\mu_{j,t}.
\end{align*}

Noticing that $\int_{M_{j,t}}\tilde{E}_{k+1}d\mu_{j,t}\leq C(n,k,\rho_-(\Omega))$, then we can choose $t$ and $j$ such that 
\begin{align*}
    &\int_{M_{j,t}}|\tilde{E}_{k+2}-c'\tilde{E}_{k+1}|d\mu_{j,t}\leq \varepsilon.
\end{align*}
Hence we can choose a sequence of smooth strictly h-convex
 bodies $\{\Omega'_j\}$ such that $\Omega'_j\rightarrow\Omega$ and 
 \begin{align*}
     \int_{\partial\Omega'_j}|\tilde{E}_{k+2}-c'\tilde{E}_{k+1}|d\mu_{j}\rightarrow 0, j\rightarrow\infty.
 \end{align*}
Then, by Theorem \ref{weakconsupport}, we have 
\begin{align*}
    \tilde{\Phi}_{n-k-2}(\Omega,\cdot)=\frac{(k+1)(n-k-1)}{(k+2)(n-k)}c\tilde{\Phi}_{n-k-1}(\Omega,\cdot).
\end{align*}
Repeating the above process multiple times, we can obtain 
$$\tilde{\Phi}_{0}(\Omega,\cdot)=c''(n,k)c\tilde{\Phi}_{1}(\Omega,\cdot).$$
Then, by Theorem \ref{arearigidity}, we obtain that $\Omega$ is a geodesic ball.
\end{proof}

\section{Weak flows and Alexandrov-Fenchel type inequalities for h-convex domains}\label{af}

We next prove Theorem \ref{weakflowthm}.

\begin{proof}[Proof of Theorem \ref{weakflowthm}]
Suppose that $\Omega\in\mathcal{H}(\mathbb{H}^{n+1})$. Let $\{\Omega_j\}$ be a sequence of smooth strictly h-convex bodies such that $\Omega_j\rightarrow\Omega$ in the Hausdorff sense as $j$ tends to $\infty$.  
Consider the flow \begin{align}
\left\{\begin{aligned} 
\frac{\partial}{\partial t} X_j(x, t) & =\left(\phi(t)-\frac{\tilde{E}_{k+1}}{\tilde{E}_k}(W)\right) \nu(x, t), \\ 
X_j(\cdot,0) & =X_{j,0}(\cdot).\end{aligned}\right.\label{smoothflow}
\end{align}
By Theorem \ref{ACWsmoothflow}, we have that, for any $j\in\mathbb{N}$,  the flow \eqref{smoothflow} exists for all time and preserves strict h-convexity.  

By Corollary \ref{yanhuajuli}, Lemma \ref{stabilityofflow}  and  arguments similar to those in \cite{Andrews2013}, we immediately obtain that the flow \eqref{flow} is well defined and 
\begin{align*}
    d_{\mathbb{H}^{n+1}}(\Omega_t,\Omega_s)\leq C(\rho_-(\Omega),n,k)(s-t+\sqrt{s-t}).
\end{align*}

 By Theorem \ref{hausdorffofhconvex} and Theorem \ref{ACWsmoothflow} we have that $\Omega_t$ is h-convex for all $t$. By the continuity of quermassintegrals with respect to the Hausdorff distance, the flow \eqref{flow} preserves the $k$-th modified Quermassintegral and makes the $k+1$-th modified Quermassintegral decrease.

By Theorem \ref{shiftedminkow} we obtain that 
$$\tilde\Phi_i(\Omega)\geq c'(\rho_-(\Omega))\tilde{\Phi}_n(\Omega)>c(\rho_-(\Omega))>0,\, 0\leq i\leq n.$$

If $\Omega$ is not a geodesic ball, by Theorem \ref{metricball} we assume that there exists a  Borel set $\beta\subseteq T^1\mathbb{H}^{n+1}$ such that 
                         $$\int_{\beta\cap\mathcal{N}_\Omega} g(v)\left(\tilde{E}_{k+1}-\frac{k+1}{n-k}\frac{\tilde\Phi_{n-k-1}(\Omega)}{\tilde\Phi_{n-k}(\Omega)}\tilde{E}_k\right)d\mathcal{L}^nv>0.$$
Then we can choose $f\in C(T^1\mathbb{H}^{n+1})$ such that $0<f<1$ and 
$$\int_{\mathcal{N}_\Omega}f g(v)\left(\tilde{E}_{k+1}-\frac{k+1}{n-k}\frac{\tilde\Phi_{n-k-1}(\Omega)}{\tilde\Phi_{n-k}(\Omega)}\tilde{E}_k\right)d\mathcal{L}^nv>0.$$

By \cite[Remark 4.1]{Zhang2026} we have $0<c_1(\rho_-(\Omega))<\int_{M_{j,t}}\tilde{E}_ld\mu<c_2(\rho_-(\Omega))$ for all $1\leq l\leq n$.
 By \cite[Equation (1.22)]{AndrewsChenWei2021}, Theorem \ref{weakconsupport}, Corollary \ref{yanhuajuli} and the argument above, there exist $J$ large enough and $0<T$ which is small enough such that
\begin{align*}-\frac{\partial}{\partial t}\tilde{W}_{k+1}(\Omega_{j,t})&=\int_{M_{j,t}}\tilde{E}_k\left(\frac{\tilde{E}_{k+1}}{\tilde{E}_k}-\phi_j(t)\right)^2 d\mu\\
&\geq \frac{1}{\int_{M_{j,t}}\tilde{E}_kd\mu}
\left(\int_{M_{j,t}}\left|\tilde{E}_k\left(\frac{\tilde{E}_{k+1}}{\tilde{E_k}}-\frac{k+1}{n-k}\frac{\tilde\Phi_{n-k-1}(\Omega_{j,t})}{\tilde\Phi_{n-k}(\Omega_{j,t})}\right)\right|d\mu\right)^2\\
&=\frac{1}{\int_{M_{j,t}}\tilde{E}_kd\mu}
\left(\int_{M_{j,t}}\left|\tilde{E}_{k+1}-\frac{k+1}{n-k}\frac{\tilde\Phi_{n-k-1}(\Omega_{j,t})}{\tilde\Phi_{n-k}(\Omega_{j,t})}\tilde{E}_k\right|d\mu\right)^2\\
&\geq C(\rho_-(\Omega))\left(\int_{M_{j,t}}\left|\tilde{E}_{k+1}-\frac{k+1}{n-k}\frac{\tilde\Phi_{n-k-1}(\Omega_{j,t})}{\tilde\Phi_{n-k}(\Omega_{j,t})}\tilde{E}_k\right|d\mu\right)^2\\
&\geq C(\rho_-(\Omega))\left(\int_{\mathcal{N}_{\Omega_{j,t}}}fg_{\Omega_{j,t}}(v)\left|\tilde{E}_{k+1}-\frac{k+1}{n-k}\frac{\tilde\Phi_{n-k-1}(\Omega_{j,t})}{\tilde\Phi_{n-k}(\Omega_{j,t})}\tilde{E}_k\right|d\mathcal{L}^n\right)^2\\
&\geq C(\rho_-(\Omega))\left(\int_{\mathcal{N}_{\Omega_{j,t}}}fg_{\Omega_{j,t}}(v)\left(\tilde{E}_{k+1}-\frac{k+1}{n-k}\frac{\tilde\Phi_{n-k-1}(\Omega_{j,t})}{\tilde\Phi_{n-k}(\Omega_{j,t})}\tilde{E}_k\right)d\mathcal{L}^n\right)^2\\
& \geq C'(\rho_-(\Omega))\left(\int_{\mathcal{N}_{\Omega}}fg_{\Omega}(v)\left(\tilde{E}_{k+1}-\frac{k+1}{n-k}\frac{\tilde\Phi_{n-k-1}(\Omega)}{\tilde\Phi_{n-k}(\Omega)}\tilde{E}_k\right)d\mathcal{L}^n\right)^2>0, 
\end{align*}
for all $j>J$ and $t\in[0,T]$. 
Hence we have $$\tilde{W}_{k+1}(\Omega_{j,t})-\tilde{W}_{k+1}(\Omega_j)<-C(\Omega)t<0,$$
from which we have that $\tilde{W}_{k+1}(\Omega_t)$ decreases strictly on $[0,T]$.

By Alexandrov reflection arguments in \cite{AndrewsChenWei2021, Zhang2026}, the set  
$\{\Omega_t\}_{t\geq 0}$ is bounded. Hence, by Theorem \ref{hausdorffofhconvex}, there exist a sequence $\{t_i\}\subseteq [0,\infty)$ and 
 an h-convex body $\Omega_{\infty}$ such that $\Omega_{t_i}\rightarrow\Omega_\infty$ as $i$ tends to $\infty$. Noticing that the $\tilde{W}_{k+1}(\Omega_t)$ is nonincreasing, there exists $W\geq \tilde{f}_{k+1}\circ \tilde{f}_k^{-1} (\tilde{W}_k(\Omega))$ such that $\tilde{W}_{k+1}(\Omega_t)\rightarrow W$ as $t\rightarrow\infty$. Hence we have $\tilde{W}_{k+1}(\Omega_\infty)=W$. If $\Omega_\infty$ is not a geodesic ball, choose a subsequence  $\{\Omega_{j_i}\}\subseteq\{\Omega_j\}$ such that 
$$d_{\mathbb{H}^{n+1}}(\Omega_{j_i,t_i},\Omega_\infty)\leq\frac{1}{i}.$$
Applying Corollary \ref{yanhuajuli} and  arguing as above, we have that there exists $I$ large enough and $0<T'$ small enough such that for any  $i>I$ there holds
$$-\frac{\partial}{\partial t}\tilde W_{k+1}(\Omega_{j_i,t_{i}+t})\geq C''(\rho_-(\Omega),\Omega_\infty)t, \,\,\text{for all $t\in[0,T']$}.$$
Hence we have that $\Omega_\infty$ is a geodesic ball. 
 In the following we assume that 
$d_{\mathbb{H}^{n+1}}(\Omega_{t_i},\Omega_\infty)$ and $d_{\mathbb{H}^{n+1}}(\Omega_{j,t_i},\Omega_\infty)$ are small enough without loss of generality. Then, by arguments in \cite[Section 6.3]{WeiYangZhou2024} and \cite[Section 5.3]{WeiYangZhou2025}, we have  
\begin{align}\label{controll2}
    &d_{\mathbb{H}^{n+1}}(\Omega_{j,t},B_{\rho_-(\Omega_{j,t})}(p_{j,t}))+d_{\mathbb{H}^{n+1}}(p_{j,t},p_\infty)\nonumber\\
    \leq& C(\rho_-(\Omega_\infty))(d_{\mathbb{H}^{n+1}}(\Omega_{j,t_i},\Omega_\infty)^\frac{1}{2}+d_{\mathbb{H}^{n+1}}(\Omega_{j,t_i},\Omega_\infty)^\frac{1}{4}),\quad t\geq t_i
\end{align}
Here $p_{j,t_i}$ is the center of the inner ball of $\Omega_{j,t_i}$ and $p_{\infty}$ is the center of the inner ball of $\Omega_{\infty}$.
Now, we first choose $i$ large enough such that $d_{\mathbb{H}^{n+1}}(\Omega_{t_i},\Omega_{\infty})$ is small enough. Noticing that the flow preserves $\tilde{W}_k$, then we can choose $j$ large enough such that $d_{\mathbb{H}^{n+1}}(\Omega_{t_i},\Omega_{j,t_i})$, $|\tilde{W}_{k}(\Omega_{j,t_i})-\tilde{W}_k(\Omega_\infty)|$ are  small enough. Hence, by \eqref{controll2}, we have that $|\rho_-(\Omega_{j,t})- \rho_-(\Omega_{\infty})|$ are small enough for all $t\in [t_i,\infty)$. Noticing that 
\begin{align*}
    d_{\mathbb{H}^{n+1}}(\Omega_{j,t},\Omega_{\infty})\leq d_{\mathbb{H}^{n+1}}(p_{j,t},p_{\infty})+|\rho_-(\Omega_{j,t})- \rho_-(\Omega_{\infty})|+d_{\mathbb{H}^{n+1}}(\Omega_{j,t},B_{\rho_-(\Omega_{j,t})}(p_{j,t})),
\end{align*} then, applying \eqref{controll2} again, we have that $d_{\mathbb{H}^{n+1}}(\Omega_{j,t},\Omega_{\infty})$ is small enough for all $t>t_i$. 
For any $\varepsilon>0$, by the argument above, we can find $t_i$ and $J$ such that for any $t>t_i$ and $j>J$ there holds
$$d_{\mathbb{H}^{n+1}}(\Omega_{j,t},\Omega_{\infty})<\varepsilon.$$
Then, for any $t>t_i$, noticing that $$d_{\mathbb{H}^{n+1}}(\Omega_{t},\Omega_{\infty})\leq d_{\mathbb{H}^{n+1}}(\Omega_{j,t},\Omega_{\infty})+d_{\mathbb{H}^{n+1}}(\Omega_{j,t},\Omega_{t})$$ and letting $j$ tend to $\infty$, we have $$d_{\mathbb{H}^{n+1}}(\Omega_{t},\Omega_{\infty})\leq\varepsilon.$$
Hence we have $\Omega_t\rightarrow\Omega_{\infty}$ as $t\rightarrow \infty$.
\end{proof}

Theorem \ref{weakflowthm} immediately implies the characterization of equality in the modified Alexandrov-Fenchel inequality in Theorem \ref{A-F}.

\section*{Acknowledgements}
	The author would like to thank Prof. Yong Wei for discussions. This work is supported by National Key Research and Development Program of China 2021YFA1001800.

\section*{Statements and Declarations}
\noindent\textbf{Competing interests.}
The author has no competing interests.

\noindent\textbf{Data availability.}
The author has no research data outside the submitted manuscript file.

\end{document}